\documentclass[a4paper,reqno]{amsart}
\usepackage{amsmath,amssymb,latexsym,tikz,color}
\usepackage{mathrsfs}
\usepackage{rotating}
\usepackage{array}
\allowdisplaybreaks

\title[Graded decomposition numbers for blocks of small weight]{Graded decomposition numbers of Ariki-Koike algebras for blocks of small weight}

\author[S.~Lyle]{Sin\'ead Lyle}
\address{School of Mathematics, University of East Anglia, Norwich NR4 7TJ, UK.}
\email{s.lyle@uea.ac.uk}
\author[O.~Ruff]{Oliver Ruff}
\address{Department of Mathematics, Kent State University at Stark, 6000 Frank Avenue NW, North Canton, OH 44720, USA.}
\email{oruff@kent.edu}

\numberwithin{equation}{section}
\numberwithin{figure}{section}
\newtheorem{lemma}{Lemma}[section]
\newtheorem{theorem}[lemma]{Theorem}
\newtheorem{proposition}[lemma]{Proposition}
\newtheorem{corollary}[lemma]{Corollary}

\newtheorem*{theorem*}{Theorem}
\newtheorem*{proposition*}{Proposition}
\newtheorem*{lemma*}{Lemma}
\newtheorem*{definition*}{Definition}
\newtheorem*{caution*}{Caution}
\newtheorem{thm}{Theorem}

\theoremstyle{definition}

\theoremstyle{remark}
\newtheorem*{remark}{Remark}
\newtheorem{ex}{Example}

\newcommand\Comment[2][\relax]{\space\par\medskip\noindent%
   \fbox{\begin{minipage}{\textwidth}\textbf{Comment\ifx\relax#1\else---#1\fi}\newline%
        #2\end{minipage}}\medskip
}

\newcommand{\h}{\mathcal{H}_{n,r}}

\newcommand{\la}{\lambda}

\newcommand{\sym}{\mathfrak{S}}
\newcommand{\Z}{\mathbb{Z}}

\newcommand{\N}{\mathbb{N}}
\newcommand{\blam}{{\boldsymbol \lambda}}
\newcommand{\bmu}{{\boldsymbol \mu}}
\newcommand{\bsig}{{\boldsymbol \sigma}}
\newcommand{\btau}{{\boldsymbol \tau}}
\newcommand{\mc}{{\boldsymbol a}}

\newcommand{\gedom}{\trianglerighteq}
\newcommand{\gdom}{\vartriangleright}
\newcommand{\Parts}{\mathscr{P}^r_n}
\newcommand{\UParts}{\mathscr{P}^r}

\def\tab(#1){\,\mbox{\tiny$\young(#1)$}\,}

\newcommand{\M}{\mathfrak{M}}
\newcommand{\B}{\mathfrak{B}}
\newcommand{\pre}{<}
\newcommand{\elp}{\hspace*{-.12cm}...}
\newcommand{\tb}{\sim_{\text{b}}}
\newcommand{\rel}{%
  \mathrel{\vbox{\offinterlineskip\ialign{%
    \hfil##\hfil\cr
    $-$\cr
    \noalign{\kern-.35ex}
    $\leadsto$\cr
}}}}

\DeclareMathOperator{\Ab}{Ab}
\DeclareMathOperator{\wt}{wt}
\DeclareMathOperator{\Min}{Min}
\DeclareMathOperator{\Rep}{Rep}
\DeclareMathOperator{\res}{res}
\DeclareMathOperator{\cha}{char}
\DeclareMathOperator{\Pt}{Pt}
\DeclareMathOperator{\In}{Ind}

\newcommand{\Ind}{\In(e)}
\newcommand{\IndOne}{\In(e+1)}

\newcolumntype{C}{>{\centering\arraybackslash$}p{0.35cm}<{$}}
\newcolumntype{D}{>{\centering\arraybackslash$}p{0.05cm}<{$}}
\newcolumntype{E}{>{\centering\arraybackslash$}p{0.15cm}<{$}}
\newcolumntype{F}{>{\centering\arraybackslash$}p{0.8cm}<{$}}

\newif\iffancy
\fancytrue

\definecolor{bead}{gray}{0.3}

\begin{document}

\begin{abstract}
We present some blocks of Ariki-Koike algebras $\h$ for which the decomposition matrices are independent of the characteristic of the underlying field.  
We complete the description of the graded decomposition numbers for blocks of Ariki-Koike algebras of weight at most two, which consists of analysing the indecomposable core blocks at level $r=3$, and give a closed formula for the decomposition numbers in this case.  
\end{abstract}

\maketitle 

\section{Introduction}
The Ariki-Koike algebras $\h$ arise naturally in multiple contexts.  Originally introduced by Ariki and Koike~\cite{ArikiKoike} as a simultaneous generalisation of the Hecke algebras of types $A$ and $B$, they also correspond to the Hecke algebras of the complex reflection groups of type $G(r,1,n)$~\cite{BroueMalle} and can be seen in the work of Cherednik~\cite{Cherednik}; furthermore, the symmetric group algebra $F\sym_n$ occurs as an example of such an algebra.  Recently, they have attracted attention due to their  relevance within the categorification program.  It was shown in the 1990s~\cite{Ariki,Grojnowski} that the finite dimensional $\h$-modules for all $n \geq 0$ categorify the irreducible highest weight module $V(\Lambda)$ of a certain Kac-Moody algebra $\mathfrak{g}$.  More recently, a result of Brundan and Kleshchev~\cite{BK:Blocks} showed that the Ariki-Koike algebras are isomorphic to certain $\Z$-graded algebras introduced independently by Khovanov and Lauda~\cite{KhovLaud:diagI,KhovLaud:diagII} and Rouquier~\cite{Rouquier}.  This defines a grading on $\h$ and gives rise to a categorification of $V(\Lambda)$ over the quantized enveloping algebra $U_v(\mathfrak{g})$.    

One important open problem in the study of the Ariki-Koike algebras is to determine the decomposition matrices, that is, find the composition factors of the Specht modules $S^\blam$.  It was shown by Brundan, Kleshchev and Wang~\cite{BKW:GradedSpecht} that the Specht modules are graded, and so we define graded decomposition numbers.  Over  fields of characteristic 0, these numbers are the polynomials arising from the LLT conjecture~\cite{LLT} and so can, in principle, be computed.  In practice, this computation is not possible, except for small values of $n$ and some specific cases where there exist closed formulae.  

Graded decomposition matrices for Ariki-Koike algebras defined over a field of positive characteristic are related to those over a field of characteristic 0 by (graded) adjustment matrices whose entries come from $\N[v^{-1}+v]$;  in particular, the graded decomposition numbers in characteristic 0 give a lower bound for those in arbitrary characteristic~\cite{BK:Blocks}.  James' conjecture~\cite{James:Ten}, for Hecke algebras of type $A$, was that when the characteristic of the field is less than the weight of the block, the adjustment matrix is the identity matrix.  However, as this paper was being written, Williamson announced  counterexamples to  Lusztig's conjecture and, consequently, to James' conjecture~\cite{Williamson2}. 

The question of when an adjustment matrix is the identity matrix is now wide open, for all $r \geq 1$, and it remains to be seen what role the weight of the block plays.  
If $r=2$ and the quantum characteristic $e$ satisfies $e=0$ or $e>n$, the decomposition numbers are independent of the characteristic of the field~\cite{BrundanStroppel},~\cite[Appendix~B]{HuMathas:QuiverSchurI}, regardless of the weight of the block.  
 In this paper, we present certain other blocks, when $r=3$ in which the decomposition numbers are also independent of the characteristic of the field.  
We first complete the description of the decomposition matrices for blocks of weight at most two.  
The blocks that we need to consider are indecomposable core blocks, with $r=3$.   (The definition of an indecomposable core block is given in Section~\ref{Defns}, while the combinatorial definitions required below appear in Section~\ref{CombDefns}.)  
We have the following simple description of the decomposition numbers for these blocks.  If $\bmu$ and $\blam$ are multipartitions in the same block, write $\bmu \leadsto \blam$ if $\blam$ is formed from $\bmu$ by removing a single rim hook from component $k$ and attaching it to component $k+1$, where $k \in \{1,2,\ldots,r-1\}$.  Write $\bmu \rel \blam$ if $\bmu \leadsto \blam$ and the leg lengths of these two hooks are equal.  If $\blam$ is a Kleshchev multipartition, let $\widetilde{\blam}=(\blam^{\diamond})'$ denote the conjugate of the image of $\blam$ under the generalisation of the Mullineux involution.  

\begin{thm} \label{MainTheorem2}
Suppose that $\blam$ and $\bmu$ are $3$-multipartitions in an indecomposable core block of weight 2 and that $\blam$ is a Kleshchev multipartition.  Then
\[d_{\bmu\blam}(v)=\begin{cases}
1, & \bmu = \blam, \\
v, & \widetilde{\blam} \gdom \bmu \gdom \blam \text{ and } (\bmu \rel \blam \text{ or } \widetilde{\blam} \rel \bmu), \\
v^2, & \bmu = \widetilde{\blam}, \\
0, & \text{otherwise}.  
\end{cases}\]
\end{thm}

This formula is close in spirit to the formulae given by Richards~\cite[Theorem~4.4]{Richards} for decomposition numbers of weight 2 blocks of Hecke algebras of type $A$ and Fayers~\cite[Theorem~3.18]{Fayers:WeightTwo} for certain weight 2 blocks of Hecke algebras of type $B$.  A faster way to compute the decomposition matrix for our blocks is given in Theorem~\ref{MainTheorem} where, for each Kleshchev multipartition $\blam$, we list the multipartitions $\bmu$ such that $D^\blam$ is a composition factor of $S^\bmu$.  In fact, the proof of Theorem~\ref{MainTheorem2} follows from case-by-case analysis of Theorem~\ref{MainTheorem}.  

For each multipartition in a core block, there is a natural way of associating a `weight graph' on $r$ vertices.  In the weight two case considered above, the graph is a line on three vertices.  A natural progression would be to consider core blocks for arbitrary values of $r$ in which the weight graphs of the corresponding multipartitions are trees; we hope to return to this work in a future paper.  

\section{Preliminaries}

\subsection{The Ariki-Koike algebras}
Let $\mathbb{F}$ be a field of arbitrary characteristic. Pick $r>0$ and $n \geq 0$. Pick $q \in \mathbb{F}$ with $ q \neq 0$, and let $e$ be minimal such that $1+q+\dots+q^{e-1}=0$, or $0$ if no such value exists.  Choose nonzero parameters $Q_1,\dots,Q_r \in \mathbb{F}$. The \emph{Ariki-Koike algebra} $\h$ is the $\mathbb{F}$-algebra with generators $T_0 , \dots, T_{n-1}$ and relations
$$\begin{array}{crcll}
& (T_i +q )(T_i -1) & = & 0 &  \text{ for } 1 \leq i \leq n-1, \\
& T_i T_j & = & T_j T_i & \text{ for } 1 \leq i,j \leq n-1, |i-j|>1, \\
& T_i T_{i+1} T_i & = & T_{i+1} T_i T_{i+1} & \text{ for } 1 \leq i \leq n-2,\\
& (T_0 - Q_1)\dots (T_0 - Q_r) & = & 0, & \\
& T_0 T_1 T_0 T_1 & = & T_1 T_0 T_1 T_0. &
\end{array}$$

We say two parameters $Q_s$ and $Q_t$ are $q$-connected if $Q_s = q^k Q_t$ for some $k \in \mathbb{Z}$.  A result of Dipper and Mathas~\cite{DM:Morita} states that each Ariki-Koike algebra~$\h$ is Morita equivalent to a direct sum of tensor products of smaller algebras whose parameters are all $q$-connected.  
In view of this result, we may assume that our parameters are $q$-connected, in fact, that they are powers of $q$.  If we set $I=\Z$ if $e=0$, and $I=\{0,1,\ldots,e-1\}$ otherwise then there exists a unique $\mc = (a_1,\ldots,a_r) \in I^r$ such that $Q_s = q^{a_s}$ for all $1 \leq s \leq r$.  We call $\mc$ the \emph{multicharge}.  

The algebras $\h$ are cellular algebras~\cite{GL,DJM:CellularBasis}, with the cell modules, also known as Specht modules, indexed by \emph{$r$-multipartitions of $n$}.  (We provide a definition of multipartitions and Kleshchev multipartitions in the next section.)  The simple modules arise as the heads of certain Specht modules which are indexed by a subset of multipartitions known as Kleshchev multipartitions.  This allows us to define the decomposition matrix of $\h$ to be the matrix recording the multiplicity of a simple module $D^\blam$ as a composition factor of a Specht module $S^\bmu$.  However, this is not the end of the story.  In~\cite{BK:Blocks}, Brundan and Kleshchev showed that the algebras $\h$ are isomorphic to certain $\Z$-graded algebras defined by Khovanov and Lauda~\cite{KhovLaud:diagI,KhovLaud:diagII} and Rouquier~\cite{Rouquier}.  The definition of these cyclotomic quiver Hecke algebras is not necessary for this paper; we refer the reader to the survey paper~\cite{Kleshchev:Survey} for more information.  We have, therefore, a grading on $\h$.  (By grading, we will always mean $\Z$-grading.)  Moreover, the Specht modules $S^\bmu$ admit a grading~\cite{BKW:GradedSpecht} and so we may talk about graded decomposition numbers.  For a graded algebra $A$, we let $\Rep(A)$ denote the category of finite-dimensional graded right $A$-modules.  Recall that if $M=\oplus_{d \in \Z} M_d \in \Rep(A)$, then for $k \in \Z$ we define $M\langle k \rangle$ to be the $A$-module isomorphic to $M$ but with grading shifted by $k$, that is, $M\langle k \rangle_d = M_{d-k}$. Given $M,L \in \Rep(A)$ with $L$ irreducible, we define the graded decomposition number 
\[[M:L]_v = \sum_{k \in \Z}[M:L\langle k \rangle]v^k\]
where $v$ is an indeterminate over $\Z$ and $[M:L\langle k \rangle]$ is the graded multiplicity of $L\langle k \rangle$ in $M$.  To see that this is well-defined, we refer the reader to~\cite{NastasescuOystaeyen} (but see also~\cite{Kleshchev:Survey}).    

Our aim is to compute certain graded decomposition numbers.  To this end, we must first introduce some combinatorics.  

\subsection{Partitions and abacus displays} \label{CombDefns}

A \emph{partition} $\la$ of $n$ is a weakly decreasing sequence of non-negative integers $(\la_1,\la_2,\dots)$, where $|\la| = \sum_{x=1}^\infty \la_x = n$. An \emph{$r$-multipartition} of $n$ (usually called a multipartition) is an $r$-tuple $\blam = (\la^{(1)},\la^{(2)},\dots,\la^{(r)})$ of partitions such that $|\la^{(1)}|  +|\la^{(2)}| + \ldots + | \la^{(r)}|  = n$.  If $\lambda$ is a partition, we define its conjugate to be the partition $\la'$ with $\la'_x = \max\{j \geq 1 \mid \lambda_j \geq i\}$ for all $x \geq 1$, and if $\blam = (\lambda^{(1)},\lambda^{(2)},\ldots,\lambda^{(r)})$ is a multipartition then we define its conjugate to be the multipartition 
$\blam' = (\lambda^{(r)'},\ldots,\lambda^{(2)'},\ldots,\lambda^{(1)'})$.
Let $\Parts$ denote the set of $r$-multipartitions of $n$ and and set $\UParts = \cup_{n \geq 0}\Parts$.  We define a partial order $\gedom$ on $\Parts$ by saying that $\bmu \gedom \blam$ if and only if
\[\sum_{s=1}^{t-1}|\mu^{(s)}| + \sum_{x=1}^z \mu^{(t)}_x \geq \sum_{s=1}^{t-1}|\lambda^{(s)}| + \sum_{x=1}^z \lambda^{(t)}_x \;\text{ for all } 1 \leq t \leq r \text{ and } z \geq 0,\]
and we say that $\bmu \gdom \blam$ if $\bmu \gedom \blam$ and $\bmu \neq \blam$.
 We make the usual identification of $\blam$ with its \emph{Young diagram}: the collection of nodes 
\[\{ (s,x,y) \mid 1 \leq s \leq r,  1 \leq y \leq \la^{(s)}_x\} \subset \{1,2,\ldots,r\} \times \Z_{>0} \times \Z_{>0}. \]
More generally, we refer to any $(s,x,y) \in \{1,2,\ldots,r\} \times \Z_{>0} \times \Z_{>0}$ as a node and say that $(s,x,y)$ is \emph{above} $(s',x',y')$ if and only if $s<s'$ or $s=s'$ and $x<x'$. 
A node is a \emph{removable}/\emph{addable} node of $\blam$ if removing it from / adding it to $\blam$ yields the Young diagram of another multipartition.  A \emph{rim hook} $\mathfrak{h}$ of $\blam$ is a connected set of nodes in a component of $\blam$ with the property that if $(s,x,y) \in \mathfrak{h}$ then $(s,x+1,y+1) \notin \mathfrak{h}$; if $|\mathfrak{h}|=h$, we call $\mathfrak{h}$ a $h$-rim hook.  Hence removing (the nodes in) a rim hook of $\blam$ yields the Young diagram of a multipartition.  We define the \emph{leg length} of the hook to be $|L|-1$, where $L=\{x \in\Z_{>0} \mid \text{ there exist } s,y \text{ such that } (s,x,y) \in \mathfrak{h}\}$; in other words, $|L|$ is the number of rows occupied by $\mathfrak{h}$.

Recall that we have fixed a multicharge $\mc=(a_1,a_2,\ldots,a_r) \in I^r$.  Define the \emph{residue} of a node $A=(s,x,y)$ by $\res(A)=a_s -x+y \mod e$, if $e>0$, and $\res(A)=a_s-x+y$ if $e=0$.  We refer to nodes of residue $i$ as \emph{$i$-nodes} and write write $c_i(\blam)$ for the number of $i$-nodes in $\blam$.  If $\blam$ is formed from $\bsig$ by adding $k>0$~nodes, all of residue $i$, then we write $\bsig \xrightarrow{i:k} \blam$ and set
\begin{multline*}
N_{i}(\bsig,\blam)=\sum_{\gamma \in \blam\setminus \bsig} 
\# \left\{ \gamma' \mid \text{$\gamma'$ an addable $i$-node of $\blam$ below $\gamma$}\right\} 
-  \# \left\{\gamma' \mid \text{$\gamma'$ a removable $i$-node of $\bsig$ below $\gamma$}\right\}.
\end{multline*}

A removable $i$-node $A$ of $\blam$ is \emph{normal} if whenever $B$ is an addable $i$-node below $A$ there are more removable $i$-nodes between $A$ and $B$ than addable $i$-nodes. The highest normal $i$-node in $\blam$ -- if such a thing exists -- is said to be \emph{good}.   A multipartition $\blam$ is called \emph{Kleshchev} if it is empty, or if another Kleshchev multipartition can be obtained from it by removing a good node (of any residue).  If $r=1$, Kleshchev is equivalent to $e$-restricted, that is, $\la_x - \la_{x+1} <e$ for all $x \geq 1$.  As previously mentioned, the Kleshchev multipartitions will index the irreducible $\h$-modules.  We define an involution $\blam \mapsto \blam^{\diamond}$ on the set of Kleshchev multipartitions as follows: Repeatedly remove good nodes from $\blam$ until the empty multipartition $\emptyset$ is reached.  Suppose the residues of those nodes were, in order of removal, $i_n,i_{n-1},\ldots,i_1$.  Form the multipartition $\blam^{\diamond}$ by adding good nodes of residue $j_1,j_2,\ldots,j_n$ in turn to $\emptyset$, where $j_s = -i_s$ if $e=0$ or $j_s =-i_s \mod e$ otherwise.  This is a generalisation of the Mullineux involution, originally defined when $r=1$~\cite{Mullineux}; see~\cite[\S 2]{Fayers:WeightII} for a discussion in terms of decomposition numbers and weight.  

It will be convenient to describe multipartitions in terms of abacus configurations.  Suppose that $\la$ is a partition and let $a \in I$.  For $j \geq 1$, set $\beta_j = \la_j-j+a $ and define 
\[\beta_a(\la)=\{\beta_j \mid j \geq 1\}\]
to be the set of $\beta$-numbers associated with the partition $\la$ and the charge $a$.  Take an abacus with runners indexed from left to right by the elements of $I$ with possible bead positions indexed by the elements of $\Z$ from top to bottom and left to right, such that for $i \in I$, the integers $k \equiv i \mod e$ appear on runner $i$.  The abacus configuration of $\la$ with respect to $a$ is then the abacus configuration with a bead at position $\beta_j$ for each $j \geq 1$.  If $e=0$, runners are indexed by the elements of $\Z$, with one bead appearing on runner $i$ for each $i \in \beta_a(\la)$.

\begin{ex}
Suppose that $e=4$.  Let $a=2$ and suppose that $\la=(10^2,5,3^3,2,1^3)$.  Then 
\[\beta_a(\la) = \{11,10,4,1,0,-1,-3,-5,-6,-7,-9,-10,\ldots\}\] and the abacus configuration has the form:

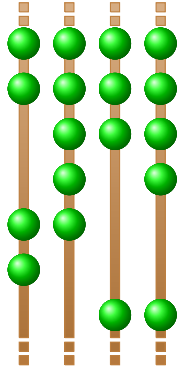
\iffancy
{\begin{center}
\begin{tikzpicture}\tikzset{yscale=0.6,xscale=0.6}
\foreach \k in {0,1,2,3} {
\shadedraw [draw=brown, bottom color=brown!90!black, top color=brown!70!white] (\k-.1,0)-- (\k+.1,0) -- (\k+0.1,6.5)--(\k-.1,6.5)--(\k-.1,0);
\filldraw[draw=brown, fill=brown!70!white](\k-.1,6.9)--(\k+.1,6.9)--(\k+.1,7.1)--(\k-.1,7.1)--(\k-.1,6.9);
\filldraw[draw=brown, fill=brown!65!white](\k-.1,7.2)--(\k+.1,7.2)--(\k+.1,7.4)--(\k-.1,7.4)--(\k-.1,7.2);
\filldraw[draw=brown, fill=brown!90!black](\k-.1,-.3)--(\k+.1,-.3)--(\k+.1,-.1)--(\k-.1,-.1)--(\k-.1,-.3);
\filldraw[draw=brown, fill=brown!95!black](\k-.1,-.6)--(\k+.1,-.6)--(\k+.1,-.4)--(\k-.1,-.4)--(\k-.1,-.6);
}
\foreach \k in {1.5,2.5,5.5,6.5} {
\shade[ball color=green] (0,\k) circle (10pt);}
\foreach \k in {2.5,3.5,4.5,5.5,6.5} {
\shade[ball color=green] (1,\k) circle (10pt);}
\foreach \k in {0.5,4.5,5.5,6.5} {
\shade[ball color=green] (2,\k) circle (10pt);}
\foreach \k in {0.5,3.5,4.5,5.5,6.5} {
\shade[ball color=green] (3,\k) circle (10pt);}
\end{tikzpicture}
\end{center}}
\else 
{
\medskip
\begin{center}
\begin{tikzpicture}\tikzset{yscale=0.6,xscale=0.6}
\foreach \k in {0,1,2,3} {
\fill [color=brown] (\k-.1,0)--(\k-.1,6.5) -- (\k+0.1,6.5)-- (\k+.1,0)--(\k-.1,0);
\fill[color=brown](\k-.1,6.9)--(\k+.1,6.9)--(\k+.1,7.1)--(\k-.1,7.1)--(\k-.1,6.9);
\fill[color=brown](\k-.1,7.2)--(\k+.1,7.2)--(\k+.1,7.4)--(\k-.1,7.4)--(\k-.1,7.2);
\fill[color=brown](\k-.1,-.1)--(\k+.1,-.1)--(\k+.1,-.3)--(\k-.1,-.3)--(\k-.1,-.1);
\fill[color=brown](\k-.1,-.4)--(\k+.1,-.4)--(\k+.1,-.6)--(\k-.1,-.6)--(\k-.1,-.4);
}
\foreach \k in {1.5,2.5,5.5,6.5} {
\fill[color=bead] (0,\k) circle (10pt);}
\foreach \k in {2.5,3.5,4.5,5.5,6.5} {
\fill[color=bead] (1,\k) circle (10pt);}
\foreach \k in {0.5,4.5,5.5,6.5} {
\fill[color=bead] (2,\k) circle (10pt);}
\foreach \k in {0.5,3.5,4.5,5.5,6.5} {
\fill[color=bead] (3,\k) circle (10pt);}
\end{tikzpicture}
\end{center}
}
\fi
\end{ex}

For a multipartition $\blam=(\la^{(1)},\la^{(2)},\ldots,\la^{(r)})$, the abacus configuration of $\blam$ is the $r$-tuple of abacuses in which abacus $s$ represents the $\beta$-numbers $\beta_{a_s}(\la^{(s)})$.  We write $\Ab(\blam)$ for the abacus configuration of $\blam$, where $\Ab(\blam)$ also determines the multicharge.  Note that increasing (resp.\! decreasing) a $\beta$-number by one in any of these sets corresponds to moving a bead across from a runner $i-1$ to runner $i$ (resp.\! from runner $i$ to runner $i-1$), considered modulo $e$, and that this corresponds to adding (resp.\! removing) a node of residue $i$ from the Young diagram of $\blam$.  Removing a $h$-rim hook from a multipartition corresponds to decreasing a $\beta$-number by $h$ in one of the sets $\beta_{a_s}(\la^{(s)})$, that is, moving a bead back $h$ spaces on the abacus, so that if $e>0$ then removing a $e$-rim hook corresponds to moving a bead up one position of the abacus.  If removing a $h$-rim hook corresponds to changing $b \in \beta_{a_s}(\la^{(s)})$ to $b-h$, then the leg length of the hook is equal to $|\{c \in \beta_{a_s}(\la^{(s)}) \mid b-h < c <h\}|$, that is the number of beads between the two positions.  

Defined thus, when $e \geq 2$ each runner of the abacus contains an infinite `sea' of beads that does not actually encode any information about the shape of the corresponding partition. If we consider a \emph{truncated} abacus configuration to be one with only finitely many beads on each runner, then we can associate it with a partition by filling in all the rows above its highest bead with other beads. Conversely, if we pick $N$ to be maximal so that $x \in \beta(\la)$ whenever $x < Ne$, we can define a canonical finite abacus configuration for $\la$ to be the one corresponding to the set $\beta_a (\la) \cap \{ Ne, Ne+1, \dots \}$. Considering this truncated abacus amounts to ignoring all the rows that do not affect the shape of the corresponding partition. 

Having set up our combinatorics, we may now return to the algebra $\h$.

\subsection{Specht modules, simple modules, and blocks} \label{Defns}
The Ariki-Koike algebras are cellular~\cite{GL,DJM:CellularBasis}, with cell modules, called \emph{Specht modules}, indexed by $r$-multipartitions of $n$, and denoted $S^\blam$ for $\blam \in \Parts$. By the general theory of cellular algebras, each $S^\blam$ comes equipped with an $\h$-invariant bilinear form, and every simple $\h$-module occurs as the quotient of some unique $S^\blam$ by the radical of this form. Write $D^\blam$ for the module obtained from $S^\blam$ in this way. 

\begin{theorem} [{\cite{Ariki:Kleshchevs}}]
Let $\blam \in \Parts$.  Then $D^\blam \neq 0$ if and only if $\blam$ is a Kleshchev multipartition. Hence, by the properties of cellular algebras,
\[\{D^{\blam} \mid \blam \text{ is a Kleshchev multipartition}\}\]
forms a complete set of pairwise non-isomorphic irreducible $\h$-modules.  
\end{theorem}

It is known that the Specht modules $S^\blam$ are graded~\cite{BKW:GradedSpecht} and therefore the (non-zero) quotient modules $D^\blam$ inherit this grading.  Given a multipartition $\bmu$ and a Kleshchev multipartition $
\blam$, define $d_{\bmu \blam}(v)$ to be the graded multiplicity of $D^\blam$ as a composition factor of $S^\bmu$.  An important open problem in the representation theory of $\h$ is to determine these composition multiplicities. Recall that $d_{\bmu\blam}(v) \in \N[v,v^{-1}]$, where we assume $0 \in \N$.  

\begin{theorem}[{\cite{GL,Ariki,BKW:GradedSpecht}}]
Suppose that $\bmu$ is a multipartition and $\blam$ a Kleshchev multipartition of $n$.  Then the following results hold.
\begin{enumerate}
\item We have $d_{\blam\blam}(v)=1$ and if $d_{\bmu\blam}(v) \neq 0$ then $\bmu \gedom \blam$.  
\item Suppose that $\mathbb{F}=\mathbb{C}$.  If $\bmu \neq \blam$ then $d_{\bmu\blam}(v) \in v\mathbb{N}[v]$.
\end{enumerate}
\end{theorem} 

We note that if $\cha(\mathbb{F}) \neq 0$ then there exist graded decomposition numbers that do not lie in $\N[v]$.  The smallest known such example, due to Evseev~\cite{Evseev}, is when $r=1, n=8$ and $e=p=2$, so for the symmetric group algebra $\mathbb{F}_{2}\sym_8$.  Setting $\la=(1^8)$ and $\mu=(3,2^2,1)$ we have $d_{\mu\la}(v)=v^{-1}+v$.

One way of computing decomposition numbers is by using $i$-induction.  We note that Proposition~\ref{iInd} below, while sufficient for our purposes, is only an application of some far more general theory.  The reader who is unfamiliar with  the generalised LLT Conjecture or Ariki's theorem is advised to either take Proposition~\ref{iInd} on trust or to consult~\cite[\S 12.1]{ArikiBook} for the notation used below. 

Assume $v$ to be an indeterminant over $\mathbb{C}$ and define the Fock space $\mathcal{F}(\mc)$ to be the $\mathbb{C}(v)$-vector space with basis $\UParts$.  Then $\mathcal{F}(\mc)$ becomes a module for the quantized enveloping algebra $U_v(\mathfrak{g})$ where $\mathfrak{g}=
\widehat{\mathfrak{sl}}_e(\mathbb{C})$ if $e>0$ and 
$\mathfrak{g}=\mathfrak{sl}_{\infty}(\mathbb{C})$ if $e=0$.  More details of this action can be found in (for example)~\cite[\S 10.2]{ArikiBook}.  In particular, the empty multipartition $\emptyset$ generates a certain highest weight module $V(\Lambda)$.  
We describe the action on $\mathcal{F}(\mc)$ of the divided powers of the Chevalley generators $F_i$, where $i \in \Z/e\Z$.  If $\bsig \in \UParts$ then for $k>0$, we have
\[F^{(k)}_i \bsig = \sum_{\bsig \xrightarrow{i:k} \blam} v^{N_i(\bsig,\blam)}\blam.\]  
For a Kleshchev multipartition $\blam \in \Parts$, define $P(\blam) = \sum_{\bmu \in \Parts}d_{\bmu\blam}(v) \bmu$. 

\begin{proposition} \label{iInd}
Suppose that $P(\bsig)=\bsig$ and that there exist $i_1,i_2,\ldots,i_m \in \Z/e\Z$ and $k_1,k_2,\ldots,k_m >0$ such that
\[F_{i_m}^{(k_m)}\ldots F_{i_2}^{(k_2)}F_{i_1}^{(k_1)}\bsig = \blam + \sum_{\bmu \neq \blam} a_{\bmu\blam}(v)\bmu\]
where $a_{\bmu\blam}(v) \in v \N[v]$ for all $\bmu$.  Then $\blam$ is a Kleshchev multipartition and $d_{\bmu\blam}(v) = a_{\bmu\blam}(v)$ for all $\bmu \neq \blam$; or equivalently, 
\[P(\blam) =  F^{(k_m)}_{i_m}\ldots F^{(k_2)}_{i_2}F^{(k_1)}_{i_1}\bsig.\]   
\end{proposition}

\begin{proof}
Suppose that $\mathbb{F}=\mathbb{C}$.  Then $F^{(k_m)}_{i_m}\ldots F^{(k_2)}_{i_2}F^{(k_1)}_{i_1}\bsig$ is a canonical basis element of the highest weight module $V(\Lambda)$ and so by Ariki's theorem~\cite[Theorem~4.4]{Ariki} (see also~\cite[Theorem~5.9]{BK:GradedDecomp}) we have that $\blam$ is a Kleshchev multipartition, that $\bmu \gdom \blam$ for all $\bmu$ such that $a_{\bmu\blam}(v) \neq 0$ and that the coefficients $a_{\bmu\blam}(v)$ give the ungraded decomposition numbers  $d_{\bmu\blam}(v)$ for $\blam \neq \bmu$ when evaluated at $v=1$.  By~\cite[Corollary 5.15]{BK:GradedDecomp}, we have the graded analogue, that is that $d_{\bmu\blam}(v) = a_{\bmu\blam}(v)$ for all $\bmu \neq \blam$.  

Now suppose that $\mathbb{F}$ has positive characteristic.  Noting that $\bmu \gdom \blam$ for all $\bmu$ such that $a_{\bmu\blam}(v) \neq 0$, Ariki's theorem shows that 
\[F^{(k_m)}_{i_m}\ldots F^{(k_2)}_{i_2}F^{(k_1)}_{i_1}P(\bsig) = P(\blam) + \sum_{\btau \neq \blam} a_\btau(v) P(\btau),\]
where the sum is over Kleshchev multipartitions $\btau \gdom \blam$ and where $a_\btau(v) \in v \mathbb{N}[v]$.  
Conversely, since in characteristic 0 we have $[S^\bmu:D^\blam]=a_{\bmu\blam}(v)$ for $\bmu \neq \blam$, the theory of graded adjustment matrices~\cite[\S 10.3]{Kleshchev:Survey} dictates that
\[P(\blam) = \blam + \sum_{\bmu \neq \blam} (a_{\bmu\blam}(v)+b_{\bmu}(v))\bmu\]
where $b_\bmu(v) \in \mathbb{N}[v,v^{-1}]$.  Combining the last two equations, we see that
\[P(\blam) = F_{i_m}^{(k_m)}\ldots F_{i_2}^{(k_2)}F_{i_1}^{(k_1)}P(\bsig) = \blam + \sum_{\bmu \neq \blam} a_{\bmu\blam}(v)\bmu\]
as required.  
\end{proof}

Since $\h$ is cellular, all composition factors of a Specht modules lie in the same block, so to classify the blocks of the $\h$, it is sufficient to decide when two Specht modules lie in the same block. 

\begin{proposition}[{\cite{LM:Blocks}}] \label{Block}
Two Specht modules $S^\blam$ and $S^\bmu$ lie in the same block if and only if $c_i (\blam) = c_i (\bmu)$ for all $i \in \Z/e\Z$. 
\end{proposition}

If $r=1$, then an alternative description of the blocks can be given as follows.  If $e=0$, the algebra is semisimple, so assume $e>0$.  Define the \emph{weight} (more precisely, $e$-weight) of a partition $\la$ to be the number of rim $e$-hooks that can be removed from it in succession.  The partition left afterwards is called the \emph{core} of $\la$, and hence partitions from which no $e$-rim hooks can be removed are called \emph{cores}.  Since removing an $e$-rim hook corresponds to moving a moving a bead up one position on the abacus, a core corresponds to an abacus configuration in which no bead can be moved any higher on its runner.  Then two Specht modules $S^\la$ and $S^\mu$ lie in the same block if and only if $\la$ and $\mu$ have the same core and the same weight.  

We now want to generalise these notions to $r \geq 1$.  Say that a multipartition is a multicore if $e=0$ or if $e \geq 2$ and the multipartition does not have any removable $e$-rim hooks.  Note that if $S^\blam$ and $S^\bmu$ lie in the same block, it is possible to have $\blam$ be a multicore, but $\bmu$ not.  We define a \emph{core block} to be a block in which all the Specht modules correspond to multicores; note that if $e=0$ then every block is a core block.  
The following definition is due to Fayers~\cite{Fayers:Weights}. We define the \emph{weight} of a multipartition $\blam$ by
\[\wt(\blam) = \sum_{s=1}^r c_{a_s}(\blam) - \frac{1}{2}\sum_{i \in G^{(k)}_{i+1} e\Z} \Big(c_i(\blam)-c_{i+1}(\blam)\Big)^2.\]
It follows from Proposition~\ref{Block} that if $S^\blam$ and $S^\bmu$ lie in the same block then $\wt(\blam)=\wt(\bmu)$ and so we will talk about the weight of a block.  This definition of `weight' agrees with the definition of the `defect' of the corresponding block of the cyclotomic quiver Hecke algebra~\cite[Eqn.~(3.10)]{Kleshchev:Survey}.  

Fayers' ingenious definition of the weight of a multipartition is not so easy to read off the abacus; in particular, it is not generally true that multicores have weight zero.  In the next section, we discuss how to find the weight of a multipartition in a core block given its abacus.  The following facts about the weight function are proved in \cite{Fayers:Weights}. 

\begin{proposition}\label{weightfacts} Let $\blam \in \Parts$. 
\begin{enumerate} 
\item We have $\wt(\blam) \geq 0$ and $\wt(\blam)=0$ if and only if $P(\blam)=\blam$. 
\item If $r=1$, $\wt(\blam)$ agrees with the usual definition of weight for partitions.
\item Let $\bmu$ be obtained from $\blam$ by removing a rim $e$-hook from any component -- equivalently, by moving a bead one space up in any component -- of $\blam$. Then $\wt(\bmu) = \wt(\blam) - r$.
\item If $\blam$ is a multicore then 
$$ \wt(\blam) = \sum_{1 \leq s < t \leq r} \wt( \la^{(s)},\la^{(t)} ).$$ 
\end{enumerate}
\end{proposition}

\begin{proof}
\begin{enumerate}
\item See~\cite[Corollary~3.9 and Theorem~4.1]{Fayers:Weights}.
\item See~\cite[Propn.~2.1]{Fayers:Weights}.
\item See~\cite[Corollary~3.4]{Fayers:Weights}.
\item See~\cite[Propn.~3.5]{Fayers:Weights}.
\end{enumerate}
\end{proof}

Using these results, the computation of $\wt(\blam)$ reduces to the case where $\blam$ is a bicore, which is also described in \cite{Fayers:Weights}.  Note that if $S^\blam$ lies in a block which is not a core block then $\wt(\blam) \geq r$.  Hence blocks of small weight often correspond to core blocks.  In fact, it is possible to further simplify the decomposition matrices corresponding to core blocks.
Say $\blam \in \Parts$ is \emph{decomposable} if there exist proper subsets $S$ and $T$ partitioning $\{1,\dots,r\}$ in such a way that $\wt(\la^{(s)},\la^{(t)}) = 0$ whenever $s \in S$ and $t \in T$. If $S = \{s_1,\dots,s_a\}$ define $\blam^S=(\la^{(s_1)},\dots ,\la^{(s_a)})$ and define $\blam^T$ analogously.  

In view of the following results of Fayers, which are proved in section 3.5 of \cite{Fayers:Cores}, it makes sense to talk about \emph{decomposable blocks}.

\begin{proposition} \label{DecompBlocks}
Let $\blam$ be a decomposable multipartition with $S$ and $T$ as in the definition above.
\begin{enumerate}
\item  We have that $\blam$ lies in a core block. 
\item If $S^\bmu$ lies in the same block as $S^\blam$ then $\bmu$ is decomposable relative to the same choice of $S$ and $T$.
\item The partition $\blam$ is Kleshchev if and only if $\blam^S$ and $\blam^T$ are both Kleshchev.
\end{enumerate}
\end{proposition}

 For $\blam$ lying in a core block, define the \emph{weight graph} $G(\blam)$ of $\blam$ to be the undirected graph with vertices $\{1,2,\ldots,r\}$ and with $\wt(\la^{(s)},\la^{(t)})$ edges between vertices $s$ and $t$.  From the definition above, it is clear that $\wt(\blam)$ is the number of edges in $G(\blam)$ and that $\blam$ is not decomposable if and only if $G(\blam)$ is connected.  

We are now in a position to consider what is known about blocks of small weight.  If a block has weight 0, it contains one Specht module $S^\blam$ with $S^\blam \cong D^\blam$.  A block of weight 1 must either satisfy $r=1$ or be a core block, where either $r=2$ or the block is decomposable with the weight graph consisting of a single edge between two vertices.  The decomposition numbers in the former case are well-known~\cite{James:Ten}.  The case that $r=2$ is dealt with explicitly in~\cite{Fayers:Weights}; however the results may easily be seen to extend to $r >2$.  

A block of weight 2 must either have $r\leq 2$ or be a core block.  The former cases have been dealt with~\cite{Richards,Fayers:Two,Fayers:WeightTwo}.  If it is a decomposable core block, then the techniques in~\cite{Fayers:Weights,Fayers:WeightTwo} enable us to determine the decomposition matrices.  The remaining open case, which we consider in the next section, is when we have an indecomposable core block, that is, when $r=3$ and the weight graph is a line.  

\subsection{Core blocks}
We finish this section with some remarks on core blocks. Recall that $I=\{0,1,\ldots,e-1\}$ if $e\geq 2$ and $I=\Z$ if $e=0$.

Let $\M$ be the set of matrices $M$ with entries in $\{0,1\}$ with $r$ rows labelled $1,2,\ldots,r$ and with columns labelled by the ordered elements of $I$, where if $e=0$ then for all $1 \leq s \leq r$, we have $M(s,i)=1$ and $M(s,j)=0$ for all $i \ll 0$ and all $j \gg 0$.  Let $\check\M \subseteq \M$ be the subset of $\M$ consisting of matrices with the further property that if $e \geq 2$ then each column contains at least one $0$.  

If $e \geq 2$, let $\B$ denote the set of tuples $B= (b_0,\dots,b_{e-1}) \in \mathbb{N}^e$ such that $b_i=0$ for at least one value of $i$ and if $e=0$, let $\B=\{0\}$.    
We call $\B$ the set of \emph{base tuples}. For $e \geq 2$, let ${\bf 0}$ denote the zero tuple $(0,\dots,0)$. 

To each pair consisting of a base tuple $B$ and a matrix $M \in \M$, we will associate an abacus configuration which in turn will give us a multicharge and a multipartition.  (We will see later that the latter lies in a core block with respect to that multicharge). That is, we set up a function 
\[\B \times \M \rightarrow I^r \times \UParts.\]


Suppose first that $e=0$ so that all multipartitions lie in core blocks.  Then there is a bijection between $I^r \times \UParts$ and $\B \times \M$: namely, given  $\mc=(a_1,a_2,\ldots,a_r) \in I^r$ and $\blam \in\UParts$, we identify them with the unique base tuple $0 \in \B$ and the matrix $M\in\M$ given by 
\[M(s,i) =\begin{cases}
1, & i \in \beta_{a_s}(\la^{(s)}), \\
0, & \text{otherwise},
\end{cases}\]
for all $1 \leq s \leq r$ and $i \in I$.   
We write $\blam = \Pt(0,M)$ (and note that $\blam$ is identified with a Young diagram along with its collection of residues, and therefore incorporates complete information about the multicharge).   Let $\prec$ be the total order on $\Z$ which agrees with the usual total order $<$.     

Now suppose $e \geq 2$.  For $B \in \B$, we define a total order $\prec$ on $\{0,1,\ldots,e-1\}$ by saying that
\[i \prec j \iff b_i < b_j \text{ or } b_i=b_j \text{ and } i<j\]
and set $\pi=\pi(B)$ to be the permutation of $\{0,1,\ldots,e-1\}$ such that
\[\pi(0) \prec \pi(1) \prec \ldots \prec \pi(e-1).\]
For $M \in \M$ and $B= (b_0,\dots,b_{e-1})\in\mathfrak{B}$, we define an $r$-tuple of truncated abaci by saying 
\[\#\{\text{Beads on runner $i$ of abacus $s$}\} = b_i + M(s,\pi^{-1}(i)).\]
If $\blam$ is the multipartition associated to this configuration and $\mc$ is the corresponding multicharge then we write $\blam =\Pt(B,M)$ (again on the understanding that information about the multicharge is incorporated in $\blam$).  By $M(s,\pi^{-1}(i))$ we mean the entry in the row indexed by $s$ and the column indexed by $\pi^{-1}(i)$, where $1 \leq s \leq r$ and $0 \leq \pi^{-1}(i) \leq e-1$.  

\begin{ex} \label{AbacusEx}
Let $e=5$ and $r=3$.  Let $B=(1,0,3,2,2)$ so that the total order $\prec$ is given by $1 \prec 0 \prec 3 \prec 4 \prec 2$
and suppose that $M=\begin{pmatrix} 
0&1&0&0&0 \\
1&0&0&1&1 \\
1&1&0&0&0 \\ 
\end{pmatrix}$.
Then the abacus configuration $\Ab(\blam)$ for $\blam=\Pt(B,M)$ is given by 
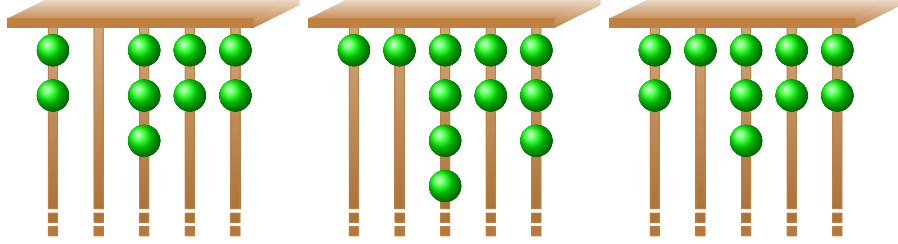
\iffancy{
\begin{center}
\begin{tikzpicture}\tikzset{yscale=0.6,xscale=0.6} 
\begin{scope}[draw=brown!20!black]
\fill[color =brown] (-1,6)--(-1,6.2)--(4.4,6.2)--(4.4,6)--cycle;
\shade[top color=brown!20!white, bottom color =brown] (0,6.7)--(-1,6.2)--(4.4,6.2)--(5.4,6.7)--cycle;
\shade[top color=brown!20!white, bottom color =brown] (4.4,6)--(4.4,6.2)--(5.4,6.7)--(5.4,6.5)--cycle;
\end{scope}
\foreach \k in {0,1,2,3,4} {
\shadedraw [draw=brown, bottom color=brown!90!black, top color=brown!70!white] (\k-.1,2)--(\k-.1,6) -- (\k+0.1,6)-- (\k+.1,2)--(\k-.1,2);
\filldraw[color=brown, fill=brown!90!black](\k-.1,1.7)--(\k-.1,1.9)--(\k+.1,1.9)--(\k+.1,1.7)--(\k-.1,1.7);
\filldraw[color=brown, fill=brown!95!black](\k-.1,1.4)--(\k-.1,1.6)--(\k+.1,1.6)--(\k+.1,1.4)--(\k-.1,1.4);
}
\foreach \k in {4.5,5.5} {
\shade[ball color=green] (0,\k) circle (10pt);}
\foreach \k in {} {
\shade[ball color=green] (1,\k) circle (10pt);}
\foreach \k in {3.5,4.5,5.5} {
\shade[ball color=green] (2,\k) circle (10pt);}
\foreach \k in {4.5,5.5} {
\shade[ball color=green] (3,\k) circle (10pt);}
\foreach \k in {4.5,5.5} {
\shade[ball color=green] (4,\k) circle (10pt);}
\end{tikzpicture}
\begin{tikzpicture}\tikzset{yscale=0.6,xscale=0.6} 
\begin{scope}[draw=brown!20!black]
\fill[color =brown] (-1,6)--(-1,6.2)--(4.4,6.2)--(4.4,6)--cycle;
\shade[top color=brown!20!white, bottom color =brown] (0,6.7)--(-1,6.2)--(4.4,6.2)--(5.4,6.7)--cycle;
\shade[top color=brown!20!white, bottom color =brown] (4.4,6)--(4.4,6.2)--(5.4,6.7)--(5.4,6.5)--cycle;
\end{scope}
\foreach \k in {0,1,2,3,4} {
\shadedraw [draw=brown, bottom color=brown!90!black, top color=brown!70!white] (\k-.1,2)--(\k-.1,6) -- (\k+0.1,6)-- (\k+.1,2)--(\k-.1,2);
\filldraw[color=brown, fill=brown!90!black](\k-.1,1.7)--(\k-.1,1.9)--(\k+.1,1.9)--(\k+.1,1.7)--(\k-.1,1.7);
\filldraw[color=brown, fill=brown!95!black](\k-.1,1.4)--(\k-.1,1.6)--(\k+.1,1.6)--(\k+.1,1.4)--(\k-.1,1.4);
}
\foreach \k in {5.5} {
\shade[ball color=green] (0,\k) circle (10pt);}
\foreach \k in {5.5} {
\shade[ball color=green] (1,\k) circle (10pt);}
\foreach \k in {2.5,3.5,4.5,5.5} {
\shade[ball color=green] (2,\k) circle (10pt);}
\foreach \k in {4.5,5.5} {
\shade[ball color=green] (3,\k) circle (10pt);}
\foreach \k in {3.5,4.5,5.5} {
\shade[ball color=green] (4,\k) circle (10pt);}
\end{tikzpicture} 
\begin{tikzpicture}\tikzset{yscale=0.6,xscale=0.6} 
\begin{scope}[draw=brown!20!black]
\fill[color =brown] (-1,6)--(-1,6.2)--(4.4,6.2)--(4.4,6)--cycle;
\shade[top color=brown!20!white, bottom color =brown] (0,6.7)--(-1,6.2)--(4.4,6.2)--(5.4,6.7)--cycle;
\shade[top color=brown!20!white, bottom color =brown] (4.4,6)--(4.4,6.2)--(5.4,6.7)--(5.4,6.5)--cycle;
\end{scope}
\foreach \k in {0,1,2,3,4} {
\shadedraw [draw=brown, bottom color=brown!90!black, top color=brown!70!white] (\k-.1,2)--(\k-.1,6) -- (\k+0.1,6)-- (\k+.1,2)--(\k-.1,2);
\filldraw[color=brown, fill=brown!90!black](\k-.1,1.7)--(\k-.1,1.9)--(\k+.1,1.9)--(\k+.1,1.7)--(\k-.1,1.7);
\filldraw[color=brown, fill=brown!95!black](\k-.1,1.4)--(\k-.1,1.6)--(\k+.1,1.6)--(\k+.1,1.4)--(\k-.1,1.4);
}
\foreach \k in {4.5,5.5} {
\shade[ball color=green] (0,\k) circle (10pt);}
\foreach \k in {5.5} {
\shade[ball color=green] (1,\k) circle (10pt);}
\foreach \k in {3.5,4.5,5.5} {
\shade[ball color=green] (2,\k) circle (10pt);}
\foreach \k in {4.5,5.5} {
\shade[ball color=green] (3,\k) circle (10pt);}
\foreach \k in {4.5,5.5} {
\shade[ball color=green] (4,\k) circle (10pt);}
\end{tikzpicture}
\end{center}
} 
\else{
\begin{center}
\begin{tikzpicture}\tikzset{yscale=0.6,xscale=0.6}
\begin{scope}[fill=brown]
\fill (-1,6)--(-1,6.2)--(4.4,6.2)--(4.4,6)--cycle;
\fill (0,6.7)--(-1,6.2)--(4.4,6.2)--(5.4,6.7)--cycle;
\fill (4.4,6)--(4.4,6.2)--(5.4,6.7)--(5.4,6.5)--cycle;
\end{scope}
\foreach \k in {0,1,2,3,4} {
\fill[color=brown] (\k-.1,2)--(\k-.1,6) -- (\k+0.1,6)-- (\k+.1,2)--(\k-.1,2);
\fill[color=brown](\k-.1,1.9)--(\k+.1,1.9)--(\k+.1,1.7)--(\k-.1,1.7)--(\k-.1,1.9);
\fill[color=brown](\k-.1,1.6)--(\k+.1,1.6)--(\k+.1,1.4)--(\k-.1,1.4)--(\k-.1,1.6);
}
\foreach \k in {4.5,5.5} {
\fill [color= bead] (0,\k) circle (10pt);}
\foreach \k in {} {
\fill[color=bead] (1,\k) circle (10pt);}
\foreach \k in {3.5,4.5,5.5} {
\fill[color=bead] (2,\k) circle (10pt);}
\foreach \k in {4.5,5.5} {
\fill[color=bead] (3,\k) circle (10pt);}
\foreach \k in {4.5,5.5} {
\fill[color= bead] (4,\k) circle (10pt);}
\end{tikzpicture}
\begin{tikzpicture}\tikzset{yscale=0.6,xscale=0.6}
\begin{scope}[draw=brown!20!black]
\fill[color =brown] (-1,6)--(-1,6.2)--(4.4,6.2)--(4.4,6)--cycle;
\fill[color =brown] (0,6.7)--(-1,6.2)--(4.4,6.2)--(5.4,6.7)--cycle;
\fill[color =brown] (4.4,6)--(4.4,6.2)--(5.4,6.7)--(5.4,6.5)--cycle;
\end{scope}
\foreach \k in {0,1,2,3,4} {
\fill [color=brown] (\k-.1,2)--(\k-.1,6) -- (\k+0.1,6)-- (\k+.1,2)--(\k-.1,2);
\fill[color=brown!90!black](\k-.1,1.9)--(\k+.1,1.9)--(\k+.1,1.7)--(\k-.1,1.7)--(\k-.1,1.9);
\fill[color=brown!95!black](\k-.1,1.6)--(\k+.1,1.6)--(\k+.1,1.4)--(\k-.1,1.4)--(\k-.1,1.6);
}
\foreach \k in {5.5} {
\fill[color=bead] (0,\k) circle (10pt);}
\foreach \k in {5.5} {
\fill[color=bead] (1,\k) circle (10pt);}
\foreach \k in {2.5,3.5,4.5,5.5} {
\fill[color=bead] (2,\k) circle (10pt);}
\foreach \k in {4.5,5.5} {
\fill[color=bead] (3,\k) circle (10pt);}
\foreach \k in {3.5,4.5,5.5} {
\fill[color=bead] (4,\k) circle (10pt);}
\end{tikzpicture}
\begin{tikzpicture}\tikzset{yscale=0.6,xscale=0.6}
\begin{scope}[fill=brown]
\fill(-1,6)--(-1,6.2)--(4.4,6.2)--(4.4,6)--cycle;
\fill(0,6.7)--(-1,6.2)--(4.4,6.2)--(5.4,6.7)--cycle;
\fill (4.4,6)--(4.4,6.2)--(5.4,6.7)--(5.4,6.5)--cycle;
\end{scope}
\foreach \k in {0,1,2,3,4} {
\fill [color=brown] (\k-.1,2)--(\k-.1,6) -- (\k+0.1,6)-- (\k+.1,2)--(\k-.1,2);
\fill[color=brown](\k-.1,1.9)--(\k+.1,1.9)--(\k+.1,1.7)--(\k-.1,1.7)--(\k-.1,1.9);
\fill[color=brown](\k-.1,1.6)--(\k+.1,1.6)--(\k+.1,1.4)--(\k-.1,1.4)--(\k-.1,1.6);
}
\foreach \k in {4.5,5.5} {
\fill[color=bead] (0,\k) circle (10pt);}
\foreach \k in {5.5} {
\fill[color=bead] (1,\k) circle (10pt);}
\foreach \k in {3.5,4.5,5.5} {
\fill[color=bead] (2,\k) circle (10pt);}
\foreach \k in {4.5,5.5} {
\fill[color=bead] (3,\k) circle (10pt);}
\foreach \k in {4.5,5.5} {
\fill[color= bead] (4,\k) circle (10pt);}
\end{tikzpicture}
\end{center}
} \fi
Note that the order of the abacus runners depends on the order $\prec$. We have $\blam=((4,2^3,1^4),(7,5,4,2^3),(3,1^3))$ and $\mc=(4,1,0)$.  Although it is not immediately obvious, it is the case that $\blam$ lies in a core block.  
\end{ex}

We now go back to the general case where $e=0$ or $e \geq 2$.  

\begin{lemma} \label{BMlamBijection}
Fix a multicharge $\mc$ and suppose $\blam \in \Parts$.  The Specht module $S^\blam$ lies in a core block if and only $\blam =\Pt(B,M)$ for some $B \in \B$ and $M \in \M$.   
\end{lemma}

\begin{proof} This is given by Theorem~3.1 in \cite{Fayers:Cores}.  
\end{proof} 

Note that if $e=0$, the pair $\Pt(B,M)$ is uniquely determined.  If $e\geq 2$ and $\blam$ lies in a core block then there is a unique pair  $B \in \B$ and $M \in \check\M$ such that $\blam=\Pt(B,M)$.  However if column $i$ of $M$ consists of entries equal to 0 then we also have $\blam=\Pt(B',M')$ for a matrix $M \in \M \setminus \check{\M}$ obtained by changing all the entries in column $i$ of $M$ to 1.  We will need such matrices later, hence why we do not restrict to choosing $M \in \check{\M}$.  

We now describe how to calculate the weight of a core block, and, given $S^\blam$ in a core block, how to describe the other Specht modules in the block.  
First suppose $M \in \M$.  For $1 \leq s,t \leq r$ define
\begin{align*}
\delta^M_{+}(t,s) & = \#\{0 \leq i \leq e-1 \mid M(t,i)=1 \text{ and } M(s,i)=0\}, \\
\delta^M_{-}(t,s) & = \#\{0 \leq i \leq e-1 \mid M(t,i)=0 \text{ and } M(s,i)=1\},\\
\wt^M(t,s)& = \Min\{\delta^M_{+}(t,s),\delta^M_{-}(t,s)\},
\end{align*}
and define
\[\wt(M) = \sum_{1 \leq s<t \leq r} \wt^M(t,s).\]

\begin{lemma} \label{WeightMatch}
Suppose $\blam =\Pt(B, M)$.  Then $\wt(\la^{(t)},\la^{(s)})=\wt^M(t,s)$ for all $1 \leq s<t \leq r$ and therefore $\wt(\blam) = \wt(M)$.
\end{lemma}

\begin{proof}
Allowing for the change of notation, this follows from~\cite[Propn.~3.8]{Fayers:Weights}.
\end{proof}

Let $M\in \M$.  Suppose there exist $i,j \in I$ and $1 \leq s,t \leq r$ with 
\[M(s,i)=M(t,j)=1, \quad M(s,j)=M(t,i)=0.\]
Let $M'$ be the matrix with 
\[M'(s,i)=M'(t,j)=0, \quad M'(s,j)=M'(t,i)=1, \quad M'(u,k)=M(u,k) \text{ otherwise}.\]
Say that $M'$ is obtained from $M$ by a \emph{bead swap}, and define the equivalence relation $\tb$ on $\M$ to be the equivalence relation generated by all bead swaps.  The next result comes from~\cite[Propn.~3.7]{Fayers:Cores}.

\begin{lemma} \label{beadswap}
Suppose that $\blam=\Pt(B,M)$.  If $\bmu$ is a multipartition then $S^\bmu$ belongs to the same core block as $S^\blam$ if and only if $\bmu:=\Pt(B,K)$ for some $K \tb M$.  
\end{lemma}

We may therefore identify the Specht modules in a core block with the elements in an equivalence class of $\M / \tb$.  The next results suggest that this will be a useful identification.  

\begin{proposition}\label{KleshReduce}
Let $e \geq 2$.  Suppose that $\blam= \Pt(B,M)$ is a partition in a core block. Then $\blam$ is Kleshchev if and only if $\Pt( {\bf 0},M)$ is Kleshchev.
\end{proposition}

\begin{proof}Write $\bmu$ for the multipartition $\Pt( {\bf 0}, M)$. It suffices to prove that we can reduce from $\blam$ to $\bmu$ by removing a sequence of good nodes. Clearly, if the highest $s$ nodes of a given residue in $\blam$ are all normal, then the multipartition $\bsig$ obtained from $\blam$ by removing them all will be Kleshchev if and only if $\blam$ is Kleshchev. Writing $B = (b_0,\dots,b_{e-1})$ as usual, we will proceed by induction on $\sum_{i=0}^{e-1} b_i \geq 0$, with the base case being trivial. 

We begin by considering the case when $b_i < b_{i+1}$ for some $i$. Then it is not possible for there to be any addable $i+1$-nodes in $\blam$, and so all removable $i+1$-nodes are normal. The multipartition $\bsig$ obtained by removing all the removable $i+1$-nodes satisfies $\bsig = \Pt(B',M)$, where $B' = (b_0, \dots, b_{i+1}, b_i, \dots, b_{e-1})$ (observe that $i$ and $i+1$ swap places in the $\prec$ ordering as we pass from $\blam$ to $\bsig$, so $M$ does not change) and $\blam$ is Kleshchev if and only if $\bsig$ is Kleshchev. By repeatedly applying this argument, we may assume that 
$$b_0 = \dots = b_j > b_{j+1} \geq \dots \geq  b_{k-1} > b_k = \dots =b_{e-1} = 0 $$   
for some $0 \leq j  < k \leq e-1$. Now there are two cases to consider. 

Firstly, if $b_0 > 1$ then there are removable $0$-nodes in $\blam$ but no addable $0$-nodes, and so we can consider the multipartition $\bsig$ obtained by removing all removable $0$-nodes from $\blam$ as above. Again, $\bsig$ will be Kleshchev if and only if $\blam$ is, and since $0$ and $e-1$ swap places in the $\prec$ ordering as we pass from $\blam$ to $\bsig$ we can write $\bsig = \Pt(B',M)$. Now, if $k <e-1$, so that $b_{e-2} = 0$, then $B' = (1,b_1,\dots,b_{e-2},b_0-1)$ and we can return to the previous case, but with fewer entries of the multicharge equal to 0, so we may repeat the arguement until we obtain $k=e-1$. If $k = e-1$, then $b_{e-2}>0$, and (since the definition of $B$ requires at least one entry to be zero) we have $B' = (0,b_1 - 1,\dots, b_{e-2}-1,b_0-2)$ and we can apply the inductive hypothesis to $\bsig$. 

Secondly (and finally) we have the case in which $b_0 = 1$ and $j=k-1$, so that $B = (1, \dots, 1, 0 ,\dots, 0)$ with the final $1$ occurring in the $j$th position. Here we can reduce to the base case: this time, let $\bsig$ be obtained from $\blam$ by moving every bead in $\blam$ $j$ positions to the left. Then $\bsig = \Pt ( { \bf 0}, M) = \bmu$, since the columns corresponding to the ones in $\blam$ now correspond to the rightmost zeros in $\bmu$, so $M$ remains unchanged. Now, since the only difference between $\bmu$ and $\blam$ is that the multicharge associated to $\bsig=\bmu$ is obtained by subtracting $k \mod e$ from each entry in the multicharge associated to $\blam$, it follows that $\blam$ is Kleshchev if and only if $\bmu$ is. \end{proof}

Define the set $\Ind$ by
\[\Ind=\begin{cases}
\Z, & e=0, \\
\{1,2,\ldots,e-1\}, & e \geq 2.
\end{cases}\]
Now let $L \in \M$.  Suppose that we have $i \in \Ind$ and $k >0$.  Write 
$L \xrightarrow{i:k} M$ 
if $M \in \M$ can be formed by choosing distinct rows $r_1,r_2,\ldots,r_k$ of $L$ with $L(r_a,i-1)=1$ and $L(r_a,i)=0$ for all $1 \leq a \leq k$ and setting
\[M(s,j) = \begin{cases}
0, & j=i-1, \, s=r_a \text{ for some } 1 \leq a \leq k, \\
1, & j=i, \, s=r_a \text{ for some } 1 \leq a \leq k, \\
L(s,j), & \text{otherwise}.  
\end{cases}\]
We can think of this as moving $k$ ones from column $i-1$ to column $i$ (and $k$ zeros back).    
If $L \xrightarrow{i:k} M$, define 
\begin{multline*}
N_{i}(L,M)=\sum_{a=1}^k \#\{ s \mid s >r_a \text{ and } M(s,i-1)=1 \text{ and } M(s,i)=0\} \\
- \#\{ s \mid s >r_a \text{ and } L(s,i-1)=0 \text{ and } L(s,i)=1\}.  
\end{multline*}
Consider the $\mathbb{C}(v)$-vector space $\mathfrak{N}$ with basis the elements of $\M$.  
For $L \in \M$, $i \in \Ind$ and $k >0$, define
\[G^{(k)}_i L = \sum_{L \xrightarrow{i:k} M} v^{N_i(L,M)} M\]
and extend linearly to give a map $G^{(k)}_i:\mathfrak{N}\rightarrow \mathfrak{N}$.    

\begin{proposition} \label{NiceInd}
Suppose that $\bsig =\Pt(B,L)$.  Let $i \in \Ind$ and $k >0$ and suppose that  
\[G^{(k)}_i L = \sum_{L \xrightarrow{i:k} M} v^{N_i(L,M)} M.\]  
Then there exist $i_1,i_2,\ldots,i_{m'} \in \Z/e\Z$ and $k_1,k_2,\ldots,k_{m'} >0$ such that 
\[F^{(k_{m'})}_{i_{m'}} \ldots F^{(k_2)}_{i_2} F^{(k_1)}_{i_1} \bsig = \sum_{L \xrightarrow{i:k} M} v^{N_i(L,M)} \Pt(B,M).\]
\end{proposition}

\begin{proof}
If $e=0$ then this follows immediately from the definitions, so suppose that $e\geq 2$. 
Let $\prec$ be the total order on $I=\{0,1,\ldots,e-1\}$ associated with the base tuple $B=(b_0,b_1,\ldots,b_{e-1})$ and $\pi$ the corresponding permutation.  
 Let $j_0 = \pi(i-1)$ and let $j_1 = \pi(i)$ so that $j_0 \prec j_1$ and there does not exist $j' \in I$ such that $j_0 \prec j' \prec j_1$. To minimise notation, we will assume that $j_0=0$; for the more general case, either make some slight modifications to the argument below or note that the combinatorics do not change if we shift the multicharge so that this holds.  Now observe that the base tuple $B$ satisfies the following properties:
\begin{itemize}
\item $b_{j_0} \leq b_{j_1}$.
\item If $j_0 < j' < j_1$ then $b_{j'}< b_{j_0}$ or $b_{j'}>b_{j_1}$.
\item If $j_1 < j' \leq e-1$ then $b_{j'}<b_{j_0}$ or $b_{j'} \geq b_{j_1}$.  
\end{itemize}
Let $\Delta=b_{j_1}-b_{j_0}$.  
Suppose that $l_0 < l_1<\ldots<l_m$ are the elements of the set \[\{l \in I \mid l < j_1 \text{ and } b_l >b_{j_1}\} \cup \{l \in I \mid l  \geq j_1 \text{ and } b_l  \geq b_{j_1}\}\cup \{j_0\},\]    
so that $j_0=l_0$ and $j_1=l_y$ for some $y$.
We claim that  
\begin{multline}  \label{IndSeqGiven}
(F^{(k(\Delta+1))}_{l_{1}} \ldots F^{(k(\Delta+1))}_{l_{0}+2} F^{(k(\Delta+1))}_{l_{0}+1})
\ldots
(F^{(k(\Delta+1))}_{l_{y}} \ldots F^{(k(\Delta+1))}_{l_{y-1}+2} F^{(k(\Delta+1))}_{l_{y-1}+1})
(F^{(k\Delta)}_{l_{y+1}} \ldots F^{(k\Delta)}_{l_{y}+2} F^{(k\Delta)}_{l_{y}+1}) \\ \ldots (F^{(k\Delta)}_{l_m} \ldots F^{(k\Delta)}_{{l_{m-1}}+2} F^{(k\Delta)}_{l_{m-1}+1})  (F^{(k\Delta)}_0 F^{(k\Delta)}_{e-1} \ldots F^{(k\Delta)}_{l_m+2} F^{(k\Delta)}_{l_m+1}) \bsig  = \sum_{L \xrightarrow{i:k} M} v^{N_i(L,M)} \Pt(B,M).
\end{multline}
Before continuing with this proof, the reader might like to examine Example~\ref{IndEx} below.  

We first investigate which abacus configurations can occur as a result of this induction sequence.  Consider 
\[(F^{(k\Delta)}_{l_{y+1}}\ldots F^{(k\Delta)}_{l_{y}+2} F^{(k\Delta)}_{l_{y}+1}) \ldots (F^{(k\Delta)}_{l_m} \ldots F^{(k\Delta)}_{{l_{m-1}}+2} F^{(k\Delta)}_{l_{m-1}+1})  (F^{(k\Delta)}_0 F^{(k\Delta)}_{e-1} \ldots F^{(k\Delta)}_{l_m+2} F^{(k\Delta)}_{l_m+1}) \bsig.\]
With the restrictions on $\Ab(\bsig)$, the abacus configurations that occur in this sum are those that can be obtained by choosing $k \Delta$ beads each of which lies on runner $j_1$ of some abacus and moving each of them forward $e-j_1$ positions to an empty position on runner 0.  Now suppose we have such a configuration, $\Ab(\btau)$ say, where $\delta_s$ such beads have been moved on abacus $s$, for $1 \leq s \leq r$,  
and consider
\begin{equation} \label{FirstInduction}
(F^{(k(\Delta+1))}_{l_{1}} \ldots F^{(k(\Delta+1))}_{l_{0}+2} F^{(k(\Delta+1))}_{l_{0}+1})
\ldots
(F^{(k(\Delta+1))}_{l_{y}} \ldots F^{(k(\Delta+1))}_{l_{y-1}+2} F^{(k(\Delta+1))}_{l_{y-1}+1}) \btau.\end{equation}
Again, this must consist of all abacus configurations that can be obtained by choosing $k(\Delta+1)$ beads each of which lies on runner $0$ of some abacus and moving each of them forward $j_1$ positions to an empty position on runner $j_1$.  On abacus $s$, we can move at most $\delta_s+1$ such beads if $L(s,i-1)=1$ and $L(s,i)=0$ and $\delta_s = \Delta$; otherwise we can move at most $\delta_s$ beads.  If it is possible to move $k(\Delta+1) = k+ \sum_{s=1}^r \delta_i$ beads, that is if Equation~\ref{FirstInduction} is not zero, there  must be $k$ values of $1 \leq s \leq r$ satisfying $L(s,i-1)=1$ and $L(s,i)=0$ and $\delta_s=\Delta$.  

So suppose that $\btau$ is such that we have $1 \leq r_1 < r_2 < \ldots < r_k \leq r$ with $L(r_a,i-1)=1$ and $L(r_a,i)=0$ and $\delta_{s_a}=\Delta$ for all $1 \leq a \leq k$.  Then $\Ab(\btau)$ is formed from $\Ab(\bsig)$ by moving the last $\Delta$ beads on runner $j_1$ of abacuses $r_1,r_2,\ldots,r_k$ along $e-j_1$ positions to runner $0$.  The only multipartition that then occurs in Equation~\ref{FirstInduction} corresponds to the multipartition formed by then moving the last $\Delta+1$ beads on runner $0$ of abacus $r_a$ along $j_1$ positions to runner $j_1$ for each $1 \leq a \leq k$.    

In short, the multipartitions which occur on the right hand side of Equation~\ref{IndSeqGiven} are exactly those formed by choosing $1 \leq r_1 < r_2 < \ldots < r_k \leq r$ with $L(r_a,i-1)=1$ and $L(r_a,i)=0$ and moving a bead in each abacus configuration $\Ab(\bsig^{(r_a)})$ from the last full position on runner 0 to the first empty position on runner $j_1$.    Hence we have shown that Equation~\ref{IndSeqGiven} is equal to
$ \sum_{L \xrightarrow{i:k} M} v^{\alpha(L,M)} \Pt(B,M)$
for some $\alpha(L,M) \in \Z$ and so it only remains to check that $\alpha(L,M)=N_i(L,M)$ for each $L \xrightarrow{i:k} M$.  Choose such an $M$, let $\blam = \Pt(B,M)$ and suppose that $\Ab(\blam)$ was formed from $\Ab(\bsig)$ by moving beads in components $r_1,r_2,\ldots,r_k$. 
Over the course of the induction, the addable and removable nodes that contribute to $s(L,M)$ almost all cancel.  If $s \neq r_{a'}$ for any $a'$ and $L(s,0)=1$ and $L(s,j_1)=1$ then each $1 \leq a \leq r$ with $r_a <s$ contributes $+1$ towards $\alpha(L,M)$.   If $s \neq r_{a'}$ for any $a'$ and $L(s,0)=0$ and $L(s,j_1)=0$ then each $1 \leq a \leq r$ with $r_a <s$ contributes $-1$ towards $\alpha(L,M)$.  All other contributions match up to sum to zero.  Hence $\alpha(L,M)=N_i(L,M)$ as required.
\end{proof}

\begin{ex}\label{IndEx}
Let $B=(1,0,3,2,2)$ and $M=\begin{pmatrix} 
0&1&0&0&0 \\
1&0&0&1&1 \\
1&1&0&0&0 \\ 
\end{pmatrix}$ as in Example~\ref{AbacusEx} and let $\bsig=\Pt(B,M)$.    Let $i=2$ and $k=1$ so that $j_0 = 0$ and $j_1=3$.  
In terms of matrices we have
\[G^{(1)}_2 \begin{pmatrix} 
0&1&0&0&0 \\
1&0&0&1&1 \\
1&1&0&0&0 \\ 
\end{pmatrix} =  \begin{pmatrix} 
0&1&0&0&0 \\
1&0&0&1&1 \\
1&0&1&0&0 \\ 
\end{pmatrix} + v
 \begin{pmatrix} 
0&0&1&0&0 \\
1&0&0&1&1 \\
1&1&0&0&0 \\ 
\end{pmatrix}.\]
Identifying multipartitions with their abacus configurations we have $\bsig$ given by  
\iffancy{
\begin{center}
\begin{tikzpicture}\tikzset{yscale=0.4,xscale=0.4}
\begin{scope}[draw=brown!20!black]
\fill[color =brown] (-1,6)--(-1,6.2)--(4.4,6.2)--(4.4,6)--cycle;
\shade[top color=brown!20!white, bottom color =brown] (0,6.7)--(-1,6.2)--(4.4,6.2)--(5.4,6.7)--cycle;
\shade[top color=brown!20!white, bottom color =brown] (4.4,6)--(4.4,6.2)--(5.4,6.7)--(5.4,6.5)--cycle;
\end{scope}
\foreach \k in {0,1,2,3,4} {
\shadedraw [draw=brown, bottom color=brown!90!black, top color=brown!70!white] (\k-.1,2)--(\k-.1,6) -- (\k+0.1,6)-- (\k+.1,2)--(\k-.1,2);
\filldraw[color=brown,fill=brown!90!black](\k-.1,1.7)--(\k-.1,1.9)--(\k+.1,1.9)--(\k+.1,1.7)--(\k-.1,1.7);
\filldraw[color=brown,fill=brown!95!black](\k-.1,1.4)--(\k-.1,1.6)--(\k+.1,1.6)--(\k+.1,1.4)--(\k-.1,1.4);
}
\foreach \k in {4.5,5.5} {
\shade[ball color=green] (0,\k) circle (10pt);}
\foreach \k in {} {
\shade[ball color=green] (1,\k) circle (10pt);}
\foreach \k in {3.5,4.5,5.5} {
\shade[ball color=green] (2,\k) circle (10pt);}
\foreach \k in {4.5,5.5} {
\shade[ball color=green] (3,\k) circle (10pt);}
\foreach \k in {4.5,5.5} {
\shade[ball color=green] (4,\k) circle (10pt);}
\end{tikzpicture}
\begin{tikzpicture}\tikzset{yscale=0.4,xscale=0.4}
\begin{scope}[draw=brown!20!black]
\fill[color =brown] (-1,6)--(-1,6.2)--(4.4,6.2)--(4.4,6)--cycle;
\shade[top color=brown!20!white, bottom color =brown] (0,6.7)--(-1,6.2)--(4.4,6.2)--(5.4,6.7)--cycle;
\shade[top color=brown!20!white, bottom color =brown] (4.4,6)--(4.4,6.2)--(5.4,6.7)--(5.4,6.5)--cycle;
\end{scope}
\foreach \k in {0,1,2,3,4} {
\shadedraw [draw=brown, bottom color=brown!90!black, top color=brown!70!white] (\k-.1,2)--(\k-.1,6) -- (\k+0.1,6)-- (\k+.1,2)--(\k-.1,2);
\filldraw[color=brown,fill=brown!90!black](\k-.1,1.7)--(\k-.1,1.9)--(\k+.1,1.9)--(\k+.1,1.7)--(\k-.1,1.7);
\filldraw[color=brown,fill=brown!95!black](\k-.1,1.4)--(\k-.1,1.6)--(\k+.1,1.6)--(\k+.1,1.4)--(\k-.1,1.4);
}
\foreach \k in {5.5} {
\shade[ball color=green] (0,\k) circle (10pt);}
\foreach \k in {5.5} {
\shade[ball color=green] (1,\k) circle (10pt);}
\foreach \k in {2.5,3.5,4.5,5.5} {
\shade[ball color=green] (2,\k) circle (10pt);}
\foreach \k in {4.5,5.5} {
\shade[ball color=green] (3,\k) circle (10pt);}
\foreach \k in {3.5,4.5,5.5} {
\shade[ball color=green] (4,\k) circle (10pt);}
\end{tikzpicture}
\begin{tikzpicture}\tikzset{yscale=0.4,xscale=0.4}
\begin{scope}[draw=brown!20!black]
\fill[color =brown] (-1,6)--(-1,6.2)--(4.4,6.2)--(4.4,6)--cycle;
\shade[top color=brown!20!white, bottom color =brown] (0,6.7)--(-1,6.2)--(4.4,6.2)--(5.4,6.7)--cycle;
\shade[top color=brown!20!white, bottom color =brown] (4.4,6)--(4.4,6.2)--(5.4,6.7)--(5.4,6.5)--cycle;
\end{scope}
\foreach \k in {0,1,2,3,4} {
\shadedraw [draw=brown, bottom color=brown!90!black, top color=brown!70!white] (\k-.1,2)--(\k-.1,6) -- (\k+0.1,6)-- (\k+.1,2)--(\k-.1,2);
\filldraw[color=brown,fill=brown!90!black](\k-.1,1.7)--(\k-.1,1.9)--(\k+.1,1.9)--(\k+.1,1.7)--(\k-.1,1.7);
\filldraw[color=brown,fill=brown!95!black](\k-.1,1.4)--(\k-.1,1.6)--(\k+.1,1.6)--(\k+.1,1.4)--(\k-.1,1.4);
}
\foreach \k in {4.5,5.5} {
\shade[ball color=green] (0,\k) circle (10pt);}
\foreach \k in {5.5} {
\shade[ball color=green] (1,\k) circle (10pt);}
\foreach \k in {3.5,4.5,5.5} {
\shade[ball color=green] (2,\k) circle (10pt);}
\foreach \k in {4.5,5.5} {
\shade[ball color=green] (3,\k) circle (10pt);}
\foreach \k in {4.5,5.5} {
\shade[ball color=green] (4,\k) circle (10pt);}
\end{tikzpicture}
\end{center}

 and $F^{(2)}_2 F^{(2)}_1 F^{(2)}_3 F^{(1)}_4 F^{(1)}_0 \bsig$ given by 

\begin{center}
\begin{tikzpicture}\tikzset{yscale=0.4,xscale=0.4}
\begin{scope}[draw=brown!20!black]
\fill[color =brown] (-1,6)--(-1,6.2)--(4.4,6.2)--(4.4,6)--cycle;
\shade[top color=brown!20!white, bottom color =brown] (0,6.7)--(-1,6.2)--(4.4,6.2)--(5.4,6.7)--cycle;
\shade[top color=brown!20!white, bottom color =brown] (4.4,6)--(4.4,6.2)--(5.4,6.7)--(5.4,6.5)--cycle;
\end{scope}
\foreach \k in {0,1,2,3,4} {
\shadedraw [draw=brown, bottom color=brown!90!black, top color=brown!70!white] (\k-.1,2)--(\k-.1,6) -- (\k+0.1,6)-- (\k+.1,2)--(\k-.1,2);
\filldraw[color=brown,fill=brown!90!black](\k-.1,1.7)--(\k-.1,1.9)--(\k+.1,1.9)--(\k+.1,1.7)--(\k-.1,1.7);
\filldraw[color=brown,fill=brown!95!black](\k-.1,1.4)--(\k-.1,1.6)--(\k+.1,1.6)--(\k+.1,1.4)--(\k-.1,1.4);
}
\foreach \k in {4.5,5.5} {
\shade[ball color=green] (0,\k) circle (10pt);}
\foreach \k in {} {
\shade[ball color=green] (1,\k) circle (10pt);}
\foreach \k in {3.5,4.5,5.5} {
\shade[ball color=green] (2,\k) circle (10pt);}
\foreach \k in {4.5,5.5} {
\shade[ball color=green] (3,\k) circle (10pt);}
\foreach \k in {4.5,5.5} {
\shade[ball color=green] (4,\k) circle (10pt);}
\end{tikzpicture}
\begin{tikzpicture}\tikzset{yscale=0.4,xscale=0.4}
\begin{scope}[draw=brown!20!black]
\fill[color =brown] (-1,6)--(-1,6.2)--(4.4,6.2)--(4.4,6)--cycle;
\shade[top color=brown!20!white, bottom color =brown] (0,6.7)--(-1,6.2)--(4.4,6.2)--(5.4,6.7)--cycle;
\shade[top color=brown!20!white, bottom color =brown] (4.4,6)--(4.4,6.2)--(5.4,6.7)--(5.4,6.5)--cycle;
\end{scope}
\foreach \k in {0,1,2,3,4} {
\shadedraw [draw=brown, bottom color=brown!90!black, top color=brown!70!white] (\k-.1,2)--(\k-.1,6) -- (\k+0.1,6)-- (\k+.1,2)--(\k-.1,2);
\filldraw[color=brown,fill=brown!90!black](\k-.1,1.7)--(\k-.1,1.9)--(\k+.1,1.9)--(\k+.1,1.7)--(\k-.1,1.7);
\filldraw[color=brown,fill=brown!95!black](\k-.1,1.4)--(\k-.1,1.6)--(\k+.1,1.6)--(\k+.1,1.4)--(\k-.1,1.4);
}
\foreach \k in {5.5} {
\shade[ball color=green] (0,\k) circle (10pt);}
\foreach \k in {5.5} {
\shade[ball color=green] (1,\k) circle (10pt);}
\foreach \k in {2.5,3.5,4.5,5.5} {
\shade[ball color=green] (2,\k) circle (10pt);}
\foreach \k in {4.5,5.5} {
\shade[ball color=green] (3,\k) circle (10pt);}
\foreach \k in {3.5,4.5,5.5} {
\shade[ball color=green] (4,\k) circle (10pt);}
\end{tikzpicture}
\begin{tikzpicture}\tikzset{yscale=0.4,xscale=0.4}
\begin{scope}[draw=brown!20!black]
\fill[color =brown] (-1,6)--(-1,6.2)--(4.4,6.2)--(4.4,6)--cycle;
\shade[top color=brown!20!white, bottom color =brown] (0,6.7)--(-1,6.2)--(4.4,6.2)--(5.4,6.7)--cycle;
\shade[top color=brown!20!white, bottom color =brown] (4.4,6)--(4.4,6.2)--(5.4,6.7)--(5.4,6.5)--cycle;
\end{scope}
\foreach \k in {0,1,2,3,4} {
\shadedraw [draw=brown, bottom color=brown!90!black, top color=brown!70!white] (\k-.1,2)--(\k-.1,6) -- (\k+0.1,6)-- (\k+.1,2)--(\k-.1,2);
\filldraw[color=brown,fill=brown!90!black](\k-.1,1.7)--(\k-.1,1.9)--(\k+.1,1.9)--(\k+.1,1.7)--(\k-.1,1.7);
\filldraw[color=brown,fill=brown!95!black](\k-.1,1.4)--(\k-.1,1.6)--(\k+.1,1.6)--(\k+.1,1.4)--(\k-.1,1.4);
}
\foreach \k in {5.5} {
\shade[ball color=green] (0,\k) circle (10pt);}
\foreach \k in {5.5} {
\shade[ball color=green] (1,\k) circle (10pt);}
\foreach \k in {3.5,4.5,5.5} {
\shade[ball color=green] (2,\k) circle (10pt);}
\foreach \k in {3.5,4.5,5.5} {
\shade[ball color=green] (3,\k) circle (10pt);}
\foreach \k in {4.5,5.5} {
\shade[ball color=green] (4,\k) circle (10pt);}
\end{tikzpicture}
\raisebox{0.8cm}{$+v$} \begin{tikzpicture}\tikzset{yscale=0.4,xscale=0.4}
\begin{scope}[draw=brown!20!black]
\fill[color =brown] (-1,6)--(-1,6.2)--(4.4,6.2)--(4.4,6)--cycle;
\shade[top color=brown!20!white, bottom color =brown] (0,6.7)--(-1,6.2)--(4.4,6.2)--(5.4,6.7)--cycle;
\shade[top color=brown!20!white, bottom color =brown] (4.4,6)--(4.4,6.2)--(5.4,6.7)--(5.4,6.5)--cycle;
\end{scope}
\foreach \k in {0,1,2,3,4} {
\shadedraw [draw=brown, bottom color=brown!90!black, top color=brown!70!white] (\k-.1,2)--(\k-.1,6) -- (\k+0.1,6)-- (\k+.1,2)--(\k-.1,2);
\filldraw[color=brown,fill=brown!90!black](\k-.1,1.7)--(\k-.1,1.9)--(\k+.1,1.9)--(\k+.1,1.7)--(\k-.1,1.7);
\filldraw[color=brown,fill=brown!95!black](\k-.1,1.4)--(\k-.1,1.6)--(\k+.1,1.6)--(\k+.1,1.4)--(\k-.1,1.4);
}
\foreach \k in {5.5} {
\shade[ball color=green] (0,\k) circle (10pt);}
\foreach \k in {} {
\shade[ball color=green] (1,\k) circle (10pt);}
\foreach \k in {3.5,4.5,5.5} {
\shade[ball color=green] (2,\k) circle (10pt);}
\foreach \k in {3.5,4.5,5.5} {
\shade[ball color=green] (3,\k) circle (10pt);}
\foreach \k in {4.5,5.5} {
\shade[ball color=green] (4,\k) circle (10pt);}
\end{tikzpicture}
\begin{tikzpicture}\tikzset{yscale=0.4,xscale=0.4}
\begin{scope}[draw=brown!20!black]
\fill[color =brown] (-1,6)--(-1,6.2)--(4.4,6.2)--(4.4,6)--cycle;
\shade[top color=brown!20!white, bottom color =brown] (0,6.7)--(-1,6.2)--(4.4,6.2)--(5.4,6.7)--cycle;
\shade[top color=brown!20!white, bottom color =brown] (4.4,6)--(4.4,6.2)--(5.4,6.7)--(5.4,6.5)--cycle;
\end{scope}
\foreach \k in {0,1,2,3,4} {
\shadedraw [draw=brown, bottom color=brown!90!black, top color=brown!70!white] (\k-.1,2)--(\k-.1,6) -- (\k+0.1,6)-- (\k+.1,2)--(\k-.1,2);
\filldraw[color=brown,fill=brown!90!black](\k-.1,1.7)--(\k-.1,1.9)--(\k+.1,1.9)--(\k+.1,1.7)--(\k-.1,1.7);
\filldraw[color=brown,fill=brown!95!black](\k-.1,1.4)--(\k-.1,1.6)--(\k+.1,1.6)--(\k+.1,1.4)--(\k-.1,1.4);
}
\foreach \k in {5.5} {
\shade[ball color=green] (0,\k) circle (10pt);}
\foreach \k in {5.5} {
\shade[ball color=green] (1,\k) circle (10pt);}
\foreach \k in {2.5,3.5,4.5,5.5} {
\shade[ball color=green] (2,\k) circle (10pt);}
\foreach \k in {4.5,5.5} {
\shade[ball color=green] (3,\k) circle (10pt);}
\foreach \k in {3.5,4.5,5.5} {
\shade[ball color=green] (4,\k) circle (10pt);}
\end{tikzpicture}
\begin{tikzpicture}\tikzset{yscale=0.4,xscale=0.4}
\begin{scope}[draw=brown!20!black]
\fill[color =brown] (-1,6)--(-1,6.2)--(4.4,6.2)--(4.4,6)--cycle;
\shade[top color=brown!20!white, bottom color =brown] (0,6.7)--(-1,6.2)--(4.4,6.2)--(5.4,6.7)--cycle;
\shade[top color=brown!20!white, bottom color =brown] (4.4,6)--(4.4,6.2)--(5.4,6.7)--(5.4,6.5)--cycle;
\end{scope}
\foreach \k in {0,1,2,3,4} {
\shadedraw [draw=brown, bottom color=brown!90!black, top color=brown!70!white] (\k-.1,2)--(\k-.1,6) -- (\k+0.1,6)-- (\k+.1,2)--(\k-.1,2);
\filldraw[color=brown,fill=brown!90!black](\k-.1,1.7)--(\k-.1,1.9)--(\k+.1,1.9)--(\k+.1,1.7)--(\k-.1,1.7);
\filldraw[color=brown,fill=brown!95!black](\k-.1,1.4)--(\k-.1,1.6)--(\k+.1,1.6)--(\k+.1,1.4)--(\k-.1,1.4);
}
\foreach \k in {4.5,5.5} {
\shade[ball color=green] (0,\k) circle (10pt);}
\foreach \k in {5.5} {
\shade[ball color=green] (1,\k) circle (10pt);}
\foreach \k in {3.5,4.5,5.5} {
\shade[ball color=green] (2,\k) circle (10pt);}
\foreach \k in {4.5,5.5} {
\shade[ball color=green] (3,\k) circle (10pt);}
\foreach \k in {4.5,5.5} {
\shade[ball color=green] (4,\k) circle (10pt);}
\end{tikzpicture}
\end{center}
}
\else{
\begin{center}
\begin{tikzpicture}\tikzset{yscale=0.4,xscale=0.4}
\begin{scope}[fill=brown]
\fill(-1,6)--(-1,6.2)--(4.4,6.2)--(4.4,6)--cycle;
\fill (0,6.7)--(-1,6.2)--(4.4,6.2)--(5.4,6.7)--cycle;
\fill(4.4,6)--(4.4,6.2)--(5.4,6.7)--(5.4,6.5)--cycle;
\end{scope}
\foreach \k in {0,1,2,3,4} {
\fill[color=brown] (\k-.1,2)--(\k-.1,6) -- (\k+0.1,6)-- (\k+.1,2)--(\k-.1,2);
\fill[color=brown](\k-.1,1.9)--(\k+.1,1.9)--(\k+.1,1.7)--(\k-.1,1.7)--(\k-.1,1.9);
\fill[color=brown](\k-.1,1.6)--(\k+.1,1.6)--(\k+.1,1.4)--(\k-.1,1.4)--(\k-.1,1.6);
}
\foreach \k in {4.5,5.5} {
\fill[color=bead] (0,\k) circle (10pt);}
\foreach \k in {} {
\fill[color=bead] (1,\k) circle (10pt);}
\foreach \k in {3.5,4.5,5.5} {
\fill[color=bead] (2,\k) circle (10pt);}
\foreach \k in {4.5,5.5} {
\fill[color=bead] (3,\k) circle (10pt);}
\foreach \k in {4.5,5.5} {
\fill[color=bead] (4,\k) circle (10pt);}
\end{tikzpicture}
\begin{tikzpicture}\tikzset{yscale=0.4,xscale=0.4}
\begin{scope}[fill=brown]
\fill(-1,6)--(-1,6.2)--(4.4,6.2)--(4.4,6)--cycle;
\fill(0,6.7)--(-1,6.2)--(4.4,6.2)--(5.4,6.7)--cycle;
\fill(4.4,6)--(4.4,6.2)--(5.4,6.7)--(5.4,6.5)--cycle;
\end{scope}
\foreach \k in {0,1,2,3,4} {
\fill[color=brown](\k-.1,2)--(\k-.1,6) -- (\k+0.1,6)-- (\k+.1,2)--(\k-.1,2);
\fill[color=brown](\k-.1,1.9)--(\k+.1,1.9)--(\k+.1,1.7)--(\k-.1,1.7)--(\k-.1,1.9);
\fill[color=brown](\k-.1,1.6)--(\k+.1,1.6)--(\k+.1,1.4)--(\k-.1,1.4)--(\k-.1,1.6);
}
\foreach \k in {5.5} {
\fill[color=bead] (0,\k) circle (10pt);}
\foreach \k in {5.5} {
\fill[color=bead] (1,\k) circle (10pt);}
\foreach \k in {2.5,3.5,4.5,5.5} {
\fill[color=bead] (2,\k) circle (10pt);}
\foreach \k in {4.5,5.5} {
\fill[color=bead] (3,\k) circle (10pt);}
\foreach \k in {3.5,4.5,5.5} {
\fill[color=bead] (4,\k) circle (10pt);}
\end{tikzpicture}
\begin{tikzpicture}\tikzset{yscale=0.4,xscale=0.4}
\begin{scope}[fill=brown]
\fill(-1,6)--(-1,6.2)--(4.4,6.2)--(4.4,6)--cycle;
\fill(0,6.7)--(-1,6.2)--(4.4,6.2)--(5.4,6.7)--cycle;
\fill(4.4,6)--(4.4,6.2)--(5.4,6.7)--(5.4,6.5)--cycle;
\end{scope}
\foreach \k in {0,1,2,3,4} {
\fill[color=brown] (\k-.1,2)--(\k-.1,6) -- (\k+0.1,6)-- (\k+.1,2)--(\k-.1,2);
\fill[color=brown](\k-.1,1.9)--(\k+.1,1.9)--(\k+.1,1.7)--(\k-.1,1.7)--(\k-.1,1.9);
\fill[color=brown](\k-.1,1.6)--(\k+.1,1.6)--(\k+.1,1.4)--(\k-.1,1.4)--(\k-.1,1.6);
}
\foreach \k in {4.5,5.5} {
\fill[color=bead] (0,\k) circle (10pt);}
\foreach \k in {5.5} {
\fill[color=bead] (1,\k) circle (10pt);}
\foreach \k in {3.5,4.5,5.5} {
\fill[color=bead] (2,\k) circle (10pt);}
\foreach \k in {4.5,5.5} {
\fill[color=bead] (3,\k) circle (10pt);}
\foreach \k in {4.5,5.5} {
\fill[color=bead] (4,\k) circle (10pt);}
\end{tikzpicture}
\end{center}
 and $F^{(2)}_2 F^{(2)}_1 F^{(2)}_3 F^{(1)}_4 F^{(1)}_0 \bsig$ given by 

\begin{center}
\begin{tikzpicture}\tikzset{yscale=0.4,xscale=0.4}
\begin{scope}[fill=brown]
\fill (-1,6)--(-1,6.2)--(4.4,6.2)--(4.4,6)--cycle;
\fill (0,6.7)--(-1,6.2)--(4.4,6.2)--(5.4,6.7)--cycle;
\fill (4.4,6)--(4.4,6.2)--(5.4,6.7)--(5.4,6.5)--cycle;
\end{scope}
\foreach \k in {0,1,2,3,4} {
\fill[color=brown] (\k-.1,2)--(\k-.1,6) -- (\k+0.1,6)-- (\k+.1,2)--(\k-.1,2);
\fill[color=brown](\k-.1,1.9)--(\k+.1,1.9)--(\k+.1,1.7)--(\k-.1,1.7)--(\k-.1,1.9);
\fill[color=brown](\k-.1,1.6)--(\k+.1,1.6)--(\k+.1,1.4)--(\k-.1,1.4)--(\k-.1,1.6);
}
\foreach \k in {4.5,5.5} {
\fill[color=bead] (0,\k) circle (10pt);}
\foreach \k in {} {
\fill[color=bead] (1,\k) circle (10pt);}
\foreach \k in {3.5,4.5,5.5} {
\fill[color=bead] (2,\k) circle (10pt);}
\foreach \k in {4.5,5.5} {
\fill[color=bead] (3,\k) circle (10pt);}
\foreach \k in {4.5,5.5} {
\fill[color=bead] (4,\k) circle (10pt);}
\end{tikzpicture}
\begin{tikzpicture}\tikzset{yscale=0.4,xscale=0.4}
\begin{scope}[fill=brown]
\fill (-1,6)--(-1,6.2)--(4.4,6.2)--(4.4,6)--cycle;
\fill(0,6.7)--(-1,6.2)--(4.4,6.2)--(5.4,6.7)--cycle;
\fill (4.4,6)--(4.4,6.2)--(5.4,6.7)--(5.4,6.5)--cycle;
\end{scope}
\foreach \k in {0,1,2,3,4} {
\fill[color=brown] (\k-.1,2)--(\k-.1,6) -- (\k+0.1,6)-- (\k+.1,2)--(\k-.1,2);
\fill[color=brown](\k-.1,1.9)--(\k+.1,1.9)--(\k+.1,1.7)--(\k-.1,1.7)--(\k-.1,1.9);
\fill[color=brown](\k-.1,1.6)--(\k+.1,1.6)--(\k+.1,1.4)--(\k-.1,1.4)--(\k-.1,1.6);
}
\foreach \k in {5.5} {
\fill[color=bead] (0,\k) circle (10pt);}
\foreach \k in {5.5} {
\fill[color=bead] (1,\k) circle (10pt);}
\foreach \k in {2.5,3.5,4.5,5.5} {
\fill[color=bead] (2,\k) circle (10pt);}
\foreach \k in {4.5,5.5} {
\fill[color=bead] (3,\k) circle (10pt);}
\foreach \k in {3.5,4.5,5.5} {
\fill[color=bead] (4,\k) circle (10pt);}
\end{tikzpicture}
\begin{tikzpicture}\tikzset{yscale=0.4,xscale=0.4}
\begin{scope}[fill=brown]
\fill(-1,6)--(-1,6.2)--(4.4,6.2)--(4.4,6)--cycle;
\fill(0,6.7)--(-1,6.2)--(4.4,6.2)--(5.4,6.7)--cycle;
\fill (4.4,6)--(4.4,6.2)--(5.4,6.7)--(5.4,6.5)--cycle;
\end{scope}
\foreach \k in {0,1,2,3,4} {
\fill[color=brown] (\k-.1,2)--(\k-.1,6) -- (\k+0.1,6)-- (\k+.1,2)--(\k-.1,2);
\fill[color=brown](\k-.1,1.9)--(\k+.1,1.9)--(\k+.1,1.7)--(\k-.1,1.7)--(\k-.1,1.9);
\fill[color=brown](\k-.1,1.6)--(\k+.1,1.6)--(\k+.1,1.4)--(\k-.1,1.4)--(\k-.1,1.6);
}
\foreach \k in {5.5} {
\fill[color=bead] (0,\k) circle (10pt);}
\foreach \k in {5.5} {
\fill[color=bead] (1,\k) circle (10pt);}
\foreach \k in {3.5,4.5,5.5} {
\fill[color=bead] (2,\k) circle (10pt);}
\foreach \k in {3.5,4.5,5.5} {
\fill[color=bead] (3,\k) circle (10pt);}
\foreach \k in {4.5,5.5} {
\fill[color=bead] (4,\k) circle (10pt);}
\end{tikzpicture} 
\raisebox{0.8cm}{$+v$} \begin{tikzpicture}\tikzset{yscale=0.4,xscale=0.4}
\begin{scope}[fill=brown]
\fill(-1,6)--(-1,6.2)--(4.4,6.2)--(4.4,6)--cycle;
\fill (0,6.7)--(-1,6.2)--(4.4,6.2)--(5.4,6.7)--cycle;
\fill (4.4,6)--(4.4,6.2)--(5.4,6.7)--(5.4,6.5)--cycle;
\end{scope}
\foreach \k in {0,1,2,3,4} {
\fill[color=brown] (\k-.1,2)--(\k-.1,6) -- (\k+0.1,6)-- (\k+.1,2)--(\k-.1,2);
\fill[color=brown](\k-.1,1.9)--(\k+.1,1.9)--(\k+.1,1.7)--(\k-.1,1.7)--(\k-.1,1.9);
\fill[color=brown](\k-.1,1.6)--(\k+.1,1.6)--(\k+.1,1.4)--(\k-.1,1.4)--(\k-.1,1.6);
}
\foreach \k in {5.5} {
\fill[color=bead] (0,\k) circle (10pt);}
\foreach \k in {} {
\fill[color=bead] (1,\k) circle (10pt);}
\foreach \k in {3.5,4.5,5.5} {
\fill[color=bead] (2,\k) circle (10pt);}
\foreach \k in {3.5,4.5,5.5} {
\fill[color=bead] (3,\k) circle (10pt);}
\foreach \k in {4.5,5.5} {
\fill[color=bead] (4,\k) circle (10pt);}
\end{tikzpicture}
\begin{tikzpicture}\tikzset{yscale=0.4,xscale=0.4}
\begin{scope}[fill=brown]
\fill (-1,6)--(-1,6.2)--(4.4,6.2)--(4.4,6)--cycle;
\fill (0,6.7)--(-1,6.2)--(4.4,6.2)--(5.4,6.7)--cycle;
\fill (4.4,6)--(4.4,6.2)--(5.4,6.7)--(5.4,6.5)--cycle;
\end{scope}
\foreach \k in {0,1,2,3,4} {
\fill[color=brown](\k-.1,2)--(\k-.1,6) -- (\k+0.1,6)-- (\k+.1,2)--(\k-.1,2);
\fill[color=brown](\k-.1,1.9)--(\k+.1,1.9)--(\k+.1,1.7)--(\k-.1,1.7)--(\k-.1,1.9);
\fill[color=brown](\k-.1,1.6)--(\k+.1,1.6)--(\k+.1,1.4)--(\k-.1,1.4)--(\k-.1,1.6);
}
\foreach \k in {5.5} {
\fill[color=bead] (0,\k) circle (10pt);}
\foreach \k in {5.5} {
\fill[color=bead] (1,\k) circle (10pt);}
\foreach \k in {2.5,3.5,4.5,5.5} {
\fill[color=bead] (2,\k) circle (10pt);}
\foreach \k in {4.5,5.5} {
\fill[color=bead] (3,\k) circle (10pt);}
\foreach \k in {3.5,4.5,5.5} {
\fill[color=bead] (4,\k) circle (10pt);}
\end{tikzpicture}
\begin{tikzpicture}\tikzset{yscale=0.4,xscale=0.4}
\begin{scope}[fill=brown]
\fill (-1,6)--(-1,6.2)--(4.4,6.2)--(4.4,6)--cycle;
\fill (0,6.7)--(-1,6.2)--(4.4,6.2)--(5.4,6.7)--cycle;
\fill (4.4,6)--(4.4,6.2)--(5.4,6.7)--(5.4,6.5)--cycle;
\end{scope}
\foreach \k in {0,1,2,3,4} {
\fill[color=brown] (\k-.1,2)--(\k-.1,6) -- (\k+0.1,6)-- (\k+.1,2)--(\k-.1,2);
\fill[color=brown](\k-.1,1.9)--(\k+.1,1.9)--(\k+.1,1.7)--(\k-.1,1.7)--(\k-.1,1.9);
\fill[color=brown](\k-.1,1.6)--(\k+.1,1.6)--(\k+.1,1.4)--(\k-.1,1.4)--(\k-.1,1.6);
}
\foreach \k in {4.5,5.5} {
\fill[color=bead] (0,\k) circle (10pt);}
\foreach \k in {5.5} {
\fill[color=bead] (1,\k) circle (10pt);}
\foreach \k in {3.5,4.5,5.5} {
\fill[color=bead] (2,\k) circle (10pt);}
\foreach \k in {4.5,5.5} {
\fill[color=bead] (3,\k) circle (10pt);}
\foreach \k in {4.5,5.5} {
\fill[color=bead] (4,\k) circle (10pt);}
\end{tikzpicture}
\end{center}
} \fi
\end{ex}

\begin{corollary} \label{PutTogether}
Suppose that $\bsig = \Pt(B,L)$ and that $\wt(\bsig)=0$.  
Suppose that there exist $i_1,i_2,\ldots,i_m \in \Ind$ and $k_1,k_2,\ldots,k_m$ with $k_l>0$ for all $l$ such that  
\[G^{(k_m)}_{i_m} \ldots G^{(k_2)}_{i_1} G^{(k_1)}_{i_1} L = K+\sum_{M \in \M \setminus \{K\}} a_{M}(v) M\] 
where $a_{M}(v) \in v\N[v]$.  Then $\blam = \Pt(B,K)$ is a Kleshchev multipartition and if $\bmu \neq \blam$ then 
\[d_{\blam\bmu}(v) = \begin{cases}
a_{M}(v), & \bmu =\Pt(B,M) \text{ for some } M \in \M, \\
0, & \text{otherwise}.  
\end{cases}\]
\end{corollary}

\begin{proof}
This follows from Proposition~\ref{iInd} and Proposition~\ref{NiceInd}.  
\end{proof}

Above, we introduced the sets $\M$ and $\check\M$.  By Lemma~\ref{BMlamBijection} and Lemma~\ref{beadswap}, the core blocks may be described using only the matrices in $\check\M$.  However, in order to state and apply Proposition~\ref{NiceInd}, we need to use the matrices in $\M$.  
We now show that when inducing, as in Corollary~\ref{PutTogether} or otherwise, we can ignore columns of $L$ in which all entries are the same.  This is intuitively obvious when the column consists of zeros, and not difficult to show when there is a column consisting of ones.

\begin{lemma} \label{Ignore01}
Suppose that $L \in \M$.  
Let $i_1,i_2,\ldots,i_m \in \Ind$ and $k_1,k_2,\ldots,k_m$ with $k_l>0$ for all $l$ and suppose  
\[G^{(k_m)}_{i_m} \ldots G^{(k_2)}_{i_1} G^{(k_1)}_{i_1} L = \sum_{M \in \M} a_{M}(v) M\] 
where $a_{M}(v) \in \N[v^{-1},v]$.    
Let $\delta \in \{0,1\}$.  Let $\M^{\delta}$ be the set of matrices obtained by adding a column consisting entirely of entries equal to $\delta$ to the same position in every matrix in $\M$, with the bijection $K \leftrightarrow \bar{K}$ for $K \in \M$.  
If $e \geq 2$ (resp. \! $e=0$) there exist $\bar{i}_{1},\bar{i}_2,\ldots,\bar{i}_{\bar{m}} \in \IndOne$ (resp. $\Ind$) and $\bar{k}_1,\bar{k}_2,\ldots,\bar{k}_{\bar m}$ with $\bar{k}_l>0$ for all $l$
such that 
\[G^{(\bar{k}_{\bar{m}})}_{\bar{i}_{\bar{m}}} \ldots G^{(\bar{k}_2)}_{\bar{i}_1} G^{(\bar{k}_1)}_{\bar{i}_1} \bar{L} = \sum_{M \in \M} a_{M}(v) \bar{M}
.\] 
\end{lemma}

\begin{proof}
Let $1 \leq l \leq m$.  If the column added is before column $i_{l}-1$, replace $G^{(k_l)}_{i_l}$ with $G^{(k_l)}_{i_l+1}$.  If the column added is between columns $i_l-1$ and $i_l$ then replace $G^{(k_l)}_{i_l}$ with $G^{(k_l)}_{i_l+1}G^{(k_l)}_{i_{l}}$ if $\delta=0$ and with $G^{(k_l)}_{i_l}G^{(k_l)}_{i_{l}+1}$ if $\delta=1$. To see that the coefficients and matrices that we obtain match up, it is sufficient to consider $G^{(k)}_i L$ for $i \in \Ind$ and $k>0$.  
Unless the column inserted is between columns $i-1$ and $i$, the lemma follows from the definitions, so assume that this is the case.  Suppose first that $\delta=0$.  Given that
\[G^{(k)}_i L =\sum_{L \xrightarrow{i:k} M} v^{N_i(L,M)} M\] 
we need to show that \[G^{(k)}_{i+1} G^{(k)}_i \bar{L} =\sum_{L \xrightarrow{i:k} M}  v^{N_i(L,M)}\bar{M}.\]
Now $G^{(k)}_{i+1} G^{(k)}_i \bar{L}$ is a sum of matrices $N$ such that there exists a matrix $J$ with $\bar{L} \xrightarrow{i:k} J$ and $J \xrightarrow{i+1:k} N$.  That means that to get $J$, we move $k$ ones from column $i-1$ of $\bar{L}$ into column $i$ of $\bar{L}$; note that column $i$ of $\bar{L}$ contains only zeros.  There is then at most one matrix $N$ such that $J \xrightarrow{i+1:k} N$: we can move all $k$ ones in column $i$ of $J$ into column $i+1$ if and only if there are zeros in each of those rows in column $i+1$.  That is, the matrices $N$ which occur are formed by moving $k$ ones from column $i-1$ to column $i+1$, so that $N=\bar{M}$ where $L  \xrightarrow{i:k} M$.  
Suppose that the entries moved are in rows $r_1,r_2,\ldots,r_k$ of $L$.  
The coefficient of $\bar{M}$ in $G^{(k)}_{i+1} G^{(k)}_i \bar{L}$ is then 
\begin{align*}
N_i(\bar{L},J)+N_{i+1}(J,\bar{M}) &= \sum_{a=1}^k \#\{ s \mid s >r_a \text{ and } J(s,i-1)=1 \text{ and } J(s,i)=0\} \\
& \hspace*{1cm} - \#\{ s \mid s >r_a \text{ and } \bar{L}(s,i-1)=0 \text{ and } \bar{L}(s,i)=1\}  \\
& \hspace*{.5cm}+ \sum_{a=1}^k \#\{ s \mid s >r_a \text{ and } \bar{M}(s,i)=1 \text{ and } \bar{M}(s,i+1)=0\} \\
& \hspace*{1cm} - \#\{ s \mid s >r_a \text{ and } J(s,i)=0 \text{ and } J(s,i+1)=1\}  \\
 &= \sum_{a=1}^k \#\{ s \mid s >r_a \text{ and } J(s,i-1)=1 \text{ and } J(s,i)=0\} \\
& \hspace*{1cm} - \#\{ s \mid s >r_a \text{ and } J(s,i)=0 \text{ and } J(s,i+1)=1\}  \\
 &= \sum_{a=1}^k \#\{ s \mid s >r_a \text{ and } J(s,i-1)=1 \text{ and } J(s,i)=0 \text{ and } J(s,i+1)=0\} \\
& \hspace*{1cm} - \#\{ s \mid s >r_a \text{ and } J(s,i-1)=0 \text{ and } J(s,i)=0 \text{ and } J(s,i+1)=1\}  \\
& =\sum_{a=1}^k \#\{ s \mid s >r_a \text{ and } M(s,i-1)=1 \text{ and } M(s,i)=0\} \\
& \hspace*{1cm}- \#\{ s \mid s >r_a \text{ and } L(s,i-1)=0 \text{ and } L(s,i)=1\} \\ 
& = N_i(L,M)\end{align*}
as required.  The proof that if $\delta=1$ then \[G^{(k)}_{i} G^{(k)}_{i+1} \bar{L} =\sum_{L \xrightarrow{i:k} M} v^{N_i(L,M)} \bar{M}\]
is then similar.  
\end{proof}

If we can find decomposition numbers using Corollary~\ref{PutTogether}, then they are independent of the characteristic of the field.  In the next sections, we demonstrate some core blocks for which this is the case.  It is not true that decomposition numbers for core blocks are always independent of the field.  For example, if $e=0$ then all blocks are core blocks.  However, if $e=0$, $r=8$, $\mc=(4,4,3,3,2,2,1,1)$ and $\blam=((2),(2),(1),(1),(3),(3),(2),(2))$ then $\dim(D^\blam)$ is is smaller in characteristic 2 than in characteristic 0~\cite{Williamson}.  (We would like to thank Andrew Mathas for his translation of Williamson's example into the notation used in this paper.)  

\section{Indecomposable weight blocks of weight two} \label{S:Weight2}
In this section, we assume that $r=3$ and consider the indecomposable core blocks of weight $2$.  The aim of the section is to prove Theorem~\ref{MainTheorem2}, that is to give a closed formula for the decomposition numbers.  We begin by describing the Specht modules that appear in these blocks and determining which of them are indexed by Kleshchev multipartitions.  We then use Corollary~\ref{PutTogether} to compute the column of the decomposition matrix corresponding to each Kleshchev multipartition; these results are given in Theorem~\ref{MainTheorem}.  Using this theorem, we are able to prove Theorem~\ref{MainTheorem2}.  

If $\blam=\Pt(B,M)$ recall the relation between $\wt(M)$ and $\wt(\blam)$ described in Lemma~\ref{WeightMatch}; in particular recall that $\wt(\blam)$ is independent of $B$.  In Proposition~\ref{IndecompTwoBlocks} below, we want to find the matrices $M \in \check\M$ such that if $\blam=\Pt(B,M)$ then $S^\blam$ belongs to an indecomposable core block of weight $2$.  We assume $e \geq 2$.  
The notation is taken to mean that $M$ is a $3 \times (x+y+z+2)$-matrix, with $x$ columns equal to the entry below $x$ and so on, and columns appearing in any order.  The subscripts on $1_y$ and $1_z$ should be ignored for the moment.  Note that the matrices indexed by high superscripts are obtained from matrices directly above by rearranging the rows.

\begin{proposition} \label{IndecompTwoBlocks}
Suppose that $e \geq 2$.  Let $\blam = \Pt(B,M)$ be such that $S^\blam$ belongs to an indecomposable core block of weight $2$, where we may assume that $M \in \check\M$.  Then $M$ has one of the following forms, for some $x,y,z \geq 0$ satisfying $x+y+z+2 = e$:  

\[\begin{array}{c|ccccc|c|cccc|c|cccc|c} \cline{2-6} \cline{8-11} \cline{13-16}
&x&z&y&1_z&1_y&& x&z+1&y&1&& x&z&y+1&1 \\ \cline{2-6} \cline{8-11} \cline{13-16}
\gamma^1 = &0&0&0&0&1 &,\quad \alpha^1= & 0&0&0&1 &,\quad \beta^1=& 0&0&0&1& ,\quad \\
&0&0&1&1&0 && 0&0&1&1 && 0&0&1&0\\
&0&1&1&0&1 && 0&1&1&0 && 0&1&1&0 \\ \cline{2-6} \cline{8-11} \cline{13-16}
\end{array}\]

\[\begin{array}{c|ccccc|c|cccc|c|cccc|c} \cline{2-6} \cline{8-11} \cline{13-16}
&x&y&z&1_y&1_z&& x&y+1&z&1&& x&y&z+1&1 \\ \cline{2-6} \cline{8-11} \cline{13-16}
\gamma^2=&0&1&1&0&1 & ,\quad \beta^2=& 0&1&1&0 &,\quad \alpha^2=& 0&1&1&0&, \\
&0&0&0&0&1 & & 0&0&0&1 && 0&0&0&1 \\
&0&0&1&1&0 && 0&0&1&1 && 0&0&1&0\\
\cline{2-6} \cline{8-11} \cline{13-16}
\end{array}\]

\[\begin{array}{c|ccccc|c|cccc|c|cccc|c} \cline{2-6} \cline{8-11} \cline{13-16}
&x&z&y&1_z&1_y&& x&z+1&y&1&& x&z&y+1&1 \\ \cline{2-6} \cline{8-11} \cline{13-16}
\gamma^3=&0&0&1&1&0 &,\quad  \alpha^3=& 0&0&1&1 &,\quad \beta^3=& 0&0&1&0&,\\
&0&1&1&0&1 && 0&1&1&0 && 0&1&1&0 \\ 
&0&0&0&0&1 & & 0&0&0&1 && 0&0&0&1 \\\cline{2-6} \cline{8-11} \cline{13-16}
\end{array}\]

\[\begin{array}{c|ccccc|c|cccc|c|cccc|c} \cline{2-6} \cline{8-11} \cline{13-16}
&x&y&z&1_y&1_z&& x&y+1&z&1&& x&y&z+1&1 \\ \cline{2-6} \cline{8-11} \cline{13-16}
\gamma^4=&0&0&1&1&0 & ,\quad \beta^4=& 0&0&1&1 &,\quad \alpha^4=& 0&0&1&0&,\\
&0&0&0&0&1 & & 0&0&0&1 && 0&0&0&1 \\
&0&1&1&0&1 && 0&1&1&0 && 0&1&1&0 \\ \cline{2-6} \cline{8-11} \cline{13-16}
\end{array}\]

\[\begin{array}{c|ccccc|c|cccc|c|cccc|c} \cline{2-6} \cline{8-11} \cline{13-16}
&x&z&y&1_z&1_y&& x&z+1&y&1&& x&z&y+1&1 \\ \cline{2-6} \cline{8-11} \cline{13-16}
\gamma^5=&0&0&0&0&1 & ,\quad \alpha^5= & 0&0&0&1 &,\quad \beta^5=& 0&0&0&1&, \\
&0&1&1&0&1 && 0&1&1&0 && 0&1&1&0 \\ 
&0&0&1&1&0 && 0&0&1&1 && 0&0&1&0\\ \cline{2-6} \cline{8-11} \cline{13-16}
\end{array}\]

\[\begin{array}{c|ccccc|c|cccc|c|cccc|c} \cline{2-6} \cline{8-11} \cline{13-16}
&x&y&z&1_y&1_z&& x&y+1&z&1&& x&y&z+1&1 \\ \cline{2-6} \cline{8-11} \cline{13-16}
\gamma^6=&0&1&1&0&1 & ,\quad \beta^6=& 0&1&1&0 &,\quad \alpha^6 = & 0&1&1&0&. \\ 
&0&0&1&1&0 && 0&0&1&1 && 0&0&1&0\\
&0&0&0&0&1 & & 0&0&0&1 && 0&0&0&1 \\
\cline{2-6} \cline{8-11} \cline{13-16}
\end{array}\]

\end{proposition}

\begin{proof} 
Suppose that $\blam=\Pt(B,M)$ is such that $S^\blam$ belongs to an indecomposable core block of weight $2$.  We write $(a \; b\; c)^T$ to denote column of length 3 in a matrix.  

Suppose that $\blam=(\la^{(1)},\la^{(2)},\la^{(3)})$ satisfies $\wt(\la^{(1)},\la^{(2)}) = \wt(\la^{(2)},\la^{(3)}) = 1$, whence $\wt(\la^{(1)},\la^{(3)}) = 0$. The latter requirement means that there cannot simultaneously be a column of $M$ with first entry 1 and last entry 0 and a column with first entry 0 and last entry 1.  
Taken in conjunction with the former requirement means that there must be $s$ columns of the form $(1\; 0\; 1)^T$ and $t$ columns of the form $(0\;1\;0)^T$, with $s,t \geq 1$ and either $s=1$ or $t=1$. If $s = t = 1$ then the condition above shows that we end up with either $\gamma^1$ or $\gamma^6$. If $s >1$ then we cannot have any columns of the form $(0\;1\;1)^T$ or $(1\;1\;0)^T$ and we cannot have both $(1\;0\;0)^T$ and $(0\;0\;1)^T$ occuring, so we end up with either $\alpha^2$ or $\alpha^4$. Similarly if $t>1$ we end up with either $\alpha^3$ or $\alpha^5$. 

All other cases now follow by permuting the rows of the matrices determined above. 
\end{proof}

It is straightforward to give an analogue of Proposition~\ref{IndecompTwoBlocks} for $e=0$: For $\delta \in \{\alpha,\beta,\gamma\}$ and $1 \leq u \leq 6$, replace matrices of the form $\delta^u$ with matrices whose rows are indexed by the elements of $\Z$ in the obvious way.  For example, replace $\gamma^1$ with a matrix which has $z \geq 0$ columns equal to $(0\;0\;1)^T$, $y \geq 0$ columns equal to $(0\;1\;1)^T$ and one column equal to $(0\;1\;0)^T$ and to $(1\;0\;1)^T$; and where all other columns are equal to $(0\;0\;0)^T$ or $(1\;1\;1)^T$ with all columns of sufficiently high index containing the former and all columns of sufficiently low index containing the latter.     

Fix once and for all a base tuple $B \in \B$.  From now on, we do not differentiate between a matrix $M \in \M$ and the partition $\Pt(B,M)$.  We will, for example, refer to a matrix $M$ as being Kleshchev whenever $\Pt(B,M)$ is Kleshchev; by Proposition~\ref{KleshReduce}, this is independent of the choice of $B$.   

Suppose we have integers $y,z \geq 0$ and a partition 
\[I= X \sqcup X' \sqcup Y \sqcup Z\] 
such that
\begin{itemize}
\item $|Y|=y+1$ and $|Z|=z+1$.
\item If $e \geq 2$ then $X'=\emptyset$.
\item If $e=0$ then $i \notin X'$ for all $i \gg 0$ and $i\in X'$ for all $i \ll 0$.  
\end{itemize}

Let $Y=\{i_1,i_2,\ldots,i_{y+1}\}$ and $Z=\{j_1,j_2,\ldots,j_{z+1}\}$ where we assume that $i_1 < i_2 \ldots < i_{y+1}$ and $j_1 < j_2 < \ldots < j_{z+1}$.   We now use the notation of Proposition~\ref{IndecompTwoBlocks}.  
Let $1 \leq u \leq 6$ and $\delta \in \{\alpha, \beta,\gamma\}$. If $\epsilon \in \{x,y,z\}$, let $\epsilon(\delta^u)$ be the column below $\epsilon$ in $\gamma^u$.  If $\epsilon \in \{1,1_y,1_z\}$, set  $\epsilon(\delta^u)$ be the column below $\epsilon$ in $\delta^u$.  Set $x'(\delta^u)=(1\;1\;1)^T$.  We define the following elements of $\M$.  
\begin{enumerate}
\item For $1 \leq k \leq y+1$, let $\alpha^u_k=\alpha^u_k(X,X',Y,Z)$ be the matrix where for $c \in I$, 
\[\text{column $c$ is equal to }
\begin{cases}
1(\alpha^u), & c = i_k, \\
y(\alpha^u), & c \in Y \text{ and } c \neq i_k, \\
z(\alpha^u), & c \in Z, \\
x(\alpha^u), & c \in X, \\
x'(\alpha^u), & c \in X'. \\
\end{cases}\]
\item For $1 \leq l \leq z+1$, let $\beta^u_l=\beta^u_l(X,X',Y,Z)$ be the matrix where
 for $c \in I$, 
\[\text{column $c$ is equal to }
\begin{cases}
y(\beta^u), & c \in Y, \\
1(\beta^u), & c = j_l, \\
z(\beta^u), & c \in Z \text{ and } c \neq j_l, \\
x(\beta^u), & c \in X, \\
x'(\beta^u), & c \in X'. \\
\end{cases}\]
\item For $1 \leq k \leq y+1$ and $1 \leq l \leq z+1$, let $\gamma^u_{kl}=\gamma^u_{kl}(X,X',Y,Z)$ be the matrix where
 for $c \in I$, 
\[\text{column $c$ is equal to }
\begin{cases}
1_y(\gamma^u), & c = i_k, \\
y(\gamma^u), & c \in Y \text{ and } c \neq i_k, \\
1_z(\gamma^u), & c = j_l, \\
z(\gamma^u), & c \in Z \text{ and } c \neq j_l, \\
x(\gamma^u), & c \in X, \\
x'(\gamma^u), & c \in X'. \\
\end{cases}\]
\end{enumerate}

\begin{proposition} \label{Partitions}
Suppose that we have an indecomposable core block of weight 2 with $r=3$.  
Then there exists $1 \leq u \leq 6$ and a partition $I= X \sqcup X' \sqcup Y \sqcup Z$ as above such that the Specht modules which appear in the block are indexed by the matrices
\[\{\alpha^u_k \mid 1 \leq k \leq y+1\} \cup
\{\beta^u_l \mid 1 \leq l \leq z+1\} \cup
\{\gamma^u_{kl} \mid 1 \leq k \leq y+1, \; 1 \leq l \leq z+1\}.\]
\end{proposition}

\begin{proof}
By Proposition~\ref{IndecompTwoBlocks}, every partition $\blam$ with $S^\blam$ lying in such a block is of the form $\delta^u$ for $\delta \in \{\alpha,\beta,\gamma\}$ and $1 \leq u \leq 6$.  By Lemma~\ref{beadswap}, the other multipartitions indexing Specht modules in the block are obtained by repeatedly applying bead swaps.  So proving the proposition is equivalent to showing that each equivalence class under the relation $\sim_b$ is of the form 
\[\{\alpha^u_k \mid 1 \leq k \leq y+1\} \cup
\{\beta^u_l \mid 1 \leq l \leq z+1\} \cup
\{\gamma^u_{kl} \mid 1 \leq k \leq y+1, \; 1 \leq l \leq z+1\},\]
for some $1 \leq u \leq 6$ and partition $I= X \sqcup X' \sqcup Y \sqcup Z$.

Fix such a partition of $I$.  Let $1 \leq k \leq y+1$ and consider the matrix $\alpha^1_k$.  Performing a single bead swap on this matrix gives us the matrices
\[\{\alpha^1_{k'} \mid 1 \leq k' \leq y+1 \text{ and } k' \neq k\}, \quad \{\beta^1_l \mid 1 \leq l \leq z+1\}, \quad \{\gamma^1_{kl} \mid 1 \leq l \leq z+1\}.\]  
Similarly if $1 \leq l \leq z+1$ then performing a single bead swap on the matrix $\beta^1_l$ gives us the matrices
\[\{\beta^1_{l'} \mid 1 \leq l' \leq z+1 \text{ and } l' \neq l\}, \quad \{\alpha^1_k \mid 1 \leq k \leq y+1\}, \quad \{\gamma^1_{kl} \mid 1 \leq k \leq y+1\}.\]  
Finally if $1 \leq k \leq y+1$ and $1 \leq l \leq z+1$ then performing a single bead swap on the matrix $\gamma^1_{kl}$ gives us the matrices
\[\{\gamma^1_{k'l} \mid 1 \leq k' \leq y+1\}, \quad \{\gamma^1_{kl'} \mid 1 \leq l' \leq z+1\}, \quad \{\alpha^1_k\}, \quad \{\beta^1_l\}.\]
So the multipartitions
\[\{\alpha^1_k \mid 1 \leq k \leq y+1\} \cup
\{\beta^1_l \mid 1 \leq l \leq z+1\} \cup
\{\gamma^1_{kl} \mid 1 \leq k \leq y+1, \; 1 \leq l \leq z+1\}\]
form an equivalence class under $\sim_b$.  
To complete the proof, note that the matrices $\delta^u$ for $u\geq 2$ are obtained by permuting the rows of the matrices $\delta^1$.  
\end{proof}

Henceforth we fix a partition $I=X \sqcup X' \sqcup Y \sqcup Z$ satisfying the conditions overleaf.  Let $Y=\{i_1,i_2,\ldots,i_{y+1}\}$ and $Z=\{j_1,j_2,\ldots,j_{z+1}\}$ where $i_1<i_2<\ldots<i_{y+1}$ and $j_1 <j_2<\ldots j_{z+1}$.  For $1 \leq u\leq 6$ define
\[\mathcal{D}^u=\{\alpha^u_k \mid 1 \leq k \leq y+1\} \cup
\{\beta^u_l \mid 1 \leq l \leq z+1\} \cup
\{\gamma^u_{kl} \mid 1 \leq k \leq y+1, \; 1 \leq l \leq z+1\}\]
to be the set of matrices indexing the Specht modules lying in the corresponding block.  We say that a matrix $M$ is Kleshchev if its Specht module corresponds to a Kleshchev multipartition.

\begin{proposition} \label{Kleshchevs}
Suppose $M \in \mathcal{D}^u$.  Then $M$ is Kleshchev if and only if it is one of the following:  
\begin{align*}
&\{\gamma^1_{kl} \mid i_k\pre j_l  \text{ and } l \neq z+1\},&
&\{\gamma^1_{kl} \mid j_l \pre i_k \text{ and } k \neq y+1\},&
&\{\alpha^1_k \mid i_k \pre j_{z+1}\},&
&\{\beta^1_l \mid j_l \pre i_{y+1}\},\\
&\{\gamma^u_{kl} \mid j_l \pre i_k \text{ or } (k \neq 1 \text{ and } l \neq z+1)\},&
&\{\alpha^u_k \mid k \neq 1 \text{ and } i_k\pre j_{z+1}\}, &
&\{\beta^u_l \mid i_1 \pre j_l\}, &
&\text{when } u=2,3, \\
&\{\gamma^u_{kl} \mid i_k \pre j_l \text{ or } (l \neq 1 \text{ and } k \neq y+1)\}&
&\{\alpha^u_k \mid k \neq y+1 \text{ and } j_1\pre i_k\},&
&\{\beta^u_l \mid j_l \pre i_{y+1}\},&
&\text{when }u=4,5,&\\
&\{\gamma^6_{kl} \mid i_k\pre j_l  \text{ and } k \neq 1\},&
&\{\gamma^6_{kl} \mid j_l \pre i_k \text{ and } l \neq 1\},&
&\{\alpha^6_k \mid j_1 \pre i_{k}\},&
&\{\beta^6_l \mid i_1 \pre j_l\}.& \\
\end{align*}
\end{proposition}

\begin{proof}

In view of Proposition~\ref{KleshReduce} it is sufficient to consider the case in which $B = {\bf 0}$, so that the relation $\prec$ is identical to $<$, and there is a maximum of one bead on each runner. This means that if $e \ge 2$, there are no removable $0$-nodes in $\blam = \Pt({\bf 0},M)$, since it is a multicore. Any removable $i$-node of $\blam$ corresponds to an entry $1$ in column $i$ of $M$ that has an entry $0$ immediately to its left (call this an ``01''), and any addable $i$-node corresponds to a $0$ in column $i$ that has a $1$ immediately to its left (likewise, a ``10''). We can read off normality from $M$: the $1$ in an ``01'' corresponds to a node which is not normal if and only if there is no ``01'' below it but at least one ``10'', or there is an ``01'' below it with a ``10'' in between them. At any given point and for any $i$, the good $i$-node is simply the normal $i$-node corresponding to a $1$ in the highest possible row. Removing a node simply corresponds to moving a $1$ one space to the left, and $\blam$ is Kleshchev if and only if we can move all the $1$s all the way to the left without ever removing any nodes which are not normal.   

We will describe the $\alpha^1_k$, $\beta^1_l$, and $\gamma^1_{kl}$ cases in detail; the others vary only in the role played by the different rows in the matrix. 

Consider the matrix $\alpha^1_k$: the condition in the proposition amounts to the requirement that there be at least one $z(\alpha^1)$ column to the right of the $1(\alpha^1)$ column. The only nodes which potentially are not normal correspond to the $1$s in the $1(\alpha^1)$ column. The nodes corresponding to these $1$s fail to be normal because they must be moved past either a $y(\alpha^1)$ or a $z(\alpha^1)$ column to their left -- and since there must be at least one $z(\alpha^1)$ column somewhere in $M$, this will definitely happen if the required condition is not met. To get rid of this problem, we need to remove nodes in such a way as to fill the $1(\alpha^1)$ column with $1$s, which is to say we need to be able to move a $1$ on the bottom row to the left, into the $1(\alpha^1)$ column. If we can do this, we can then move all the $1$s all the way to the left, by moving leftmost columns first and working from the top row to the bottom. 

For the matrix $\beta^1_l$, initially, the only node which potentially is not normal corresponds to the $1$ entry in the $1(\beta^1)$ column. The condition in the proposition is equivalent to the assertion that there is a $y(\beta^1)$ column somewhere to the right of the $1(\beta^1)$ column. If this is the case, then we can fill the $1(\alpha^1)$ column with $1$s by removing normal nodes, and from that configuration all the $1$s can be moved all the way to the left, as in the previous paragraph. On the other hand, if $i_{y+1} < j_l$, so there is no $y(\beta^1)$ column to the right of the $1(\beta^1)$ column but at least one to its left, then it is impossible to move the $1$ in the top row all the way to the left -- there is always a ``10'' somewhere in the middle row to block it. 

For the matrix $\gamma^1_{kl}$, the condition in the proposition says that if the $1_z(\gamma^1)$ column is to the right of the $1_y(\gamma^1)$ column then there must be at least one $z(\gamma^1)$ column further to its right, and if the $1_y(\gamma^1)$ column is to the right of the $1_z(\gamma^1)$ column then there must be a $y(\gamma^1)$ column further to its right. Either way, the $1_y(\gamma^1)$ and $1_z(\gamma^1)$ columns contain the $1$s corresponding to the only nodes which potentially are not normal, and $\blam$ will once again be Kleshchev if we can manage to fill the leftmost of these columns with $1$s. If the $1_y(\gamma^1)$ column occurs to the right of the $1_z(\gamma^1)$ column then at some stage we need to move a $1$ corresponding to a node which is not normal in the middle row; provided that there is a $1$ in the bottom row somewhere to its right, we can move it until it is below this $1$, at which point the previously addable node that was causing the abnormality will no longer be addable. (That is, the ``10'' in the bottom row will have become a ``11.'')  On the other hand, if the $1_z(\gamma^1)$ column occurs to the right of the $1_y(\gamma^1)$ column then at some stage we first to move a $1$ corresponding to a node which is not normal in the top row -- provided there is a $z(\gamma^1)$ somewhere to its right, we can move both its $1$s over until this $1$ is rendered normal, then we need to move the $1$ in the bottom row, which must be normal (but will only be good once the $1$ in the top row is dealt with).  Either way, once the leftmost of the $1_y(\gamma^1)$ and $1_z(\gamma^1)$ columns has been filled with $1$s, we can proceed as in the $\alpha^1_k$ and $\beta^1_l$ cases.  \end{proof}

\begin{remark}
Combining Proposition~\ref{Partitions} and Proposition~\ref{Kleshchevs}, a simple counting argument shows that there are $(y+2)(z+2)-1$ Specht modules and $(y+1)(z+1)$ simple modules in each block.
\end{remark}

We are now ready to compute the decomposition matrices.  Recall that if $\blam$ is a Kleshchev multipartition, we defined
\[P(\blam) = \sum_{\bmu \in \Parts}d_{\bmu\blam}(v) \bmu.\] 
If $e\geq 2$, set $i_{0} = j_{0} = -1$ and $i_{y+2}=j_{z+2}=e$.  If $e=0$, choose $i_0,j_0 \ll 0$ and $i_{y+2},j_{z+2} \gg 0$.  We make these definitions so that the indices $\tilde{k},\tilde{l}$ below always exist.    

We now find $P(\delta)$ for all Kleshchev multipartitions $\delta \in \mathcal{D}^u$, for each $1 \leq u \leq 6$.  

\begin{theorem} \label{MainTheorem} \quad 

\begin{enumerate}
\item
\begin{itemize}
\item Suppose that $\alpha^1_k$ is Kleshchev.  Let $\tilde{k}$ be minimal such that $i_k \pre j_{\tilde{k}}$.  Then
\[P(\alpha^1_k) = 
\begin{cases}
\alpha^1_k + v \gamma^1_{k \tilde{k}} +v^2 \beta^1_{\tilde{k}}, &  j_{\tilde{k}}\pre i_{k+1}, \\
\alpha^1_k + v \alpha^1_{k+1}+ v \gamma^1_{k,\tilde{k}} + v^2 \gamma^1_{k+1,\tilde{k}}, &  i_{k+1} \pre j_{\tilde{k}}. 
\end{cases}\]
\item Suppose that $\beta^1_l$ is Kleshchev.  Let $\tilde{l}$ be minimal such that $j_l \pre i_{\tilde{l}}$.  Then
\[P(\beta^1_l) = \begin{cases}
\beta^1_l + v \gamma^1_{\tilde{l},l}  + v^2 \alpha^1_{\tilde{l}}, &  i_{\tilde{l}} \pre j_{l+1}, \\
\beta^1_l+v \beta^1_{l+1} + v\gamma^1_{\tilde{l},l} + v^2 \gamma^1_{\tilde{l},l+1}, & j_{l+1} \pre i_{\tilde{l}}.
\end{cases}\]
\item Suppose that $\gamma^1_{k,l}$ is Kleshchev.  Then
\[P(\gamma^1_{k,l}) = \begin{cases}
\gamma^1_{k,l} + v \gamma^1_{k,l+1} + v \gamma^1_{k+1,l} +v^2 \gamma^1_{k+1,l+1}, &  i_{k+1} \pre j_l \text{ or } j_{l+1} \pre i_k, \\
\gamma^1_{k,l}+ v\gamma^1_{k,l+1} + v\alpha_{k+1}^1 + v\beta^1_l+ v^2 \gamma^1_{k+1,l+1}  ,  
& i_k \pre j_l \pre i_{k+1} \pre j_{l+1}, \\
\gamma^1_{k,l} +v\gamma^1_{k,l+1} + v\beta^1_l + v^2 \beta^1_{l+1}, 
&i_k \pre j_l \pre  j_{l+1} \pre i_{k+1}, \\
\gamma^1_{k,l} +v\gamma^1_{k+1,l} + v\alpha^1_k +v^2 \alpha^1_{k+1}, 
& j_l \pre i_k \pre i_{k+1} \pre j_{l+1}, \\
\gamma^1_{k,l}+ v\gamma^1_{k+1,l}  + v\alpha^1_k + v\beta^1_{l+1} + v^2 \gamma^1_{k+1,l+1},  
& j_l \pre i_k \pre j_{l+1} \pre i_{k+1}. \\
\end{cases}\]
\end{itemize}
\item Suppose that $u=2$ or $u=3$.  
\begin{itemize}  
\item Suppose that $\alpha^u_k$ is Kleshchev. 
Let $\tilde{k}$ be minimal such that $i_k \pre j_{\tilde{k}}$.  Then
\[P(\alpha^u_k) = 
\begin{cases}
\alpha^u_k + v\beta^u_{\tilde{k}-1} + v\gamma^u_{k \tilde{k}}  + v\gamma^u_{k-1 \tilde{k}-1} + v^2 \gamma^u_{k-1 \tilde{k}}, & i_{k-1} \pre j_{\tilde{k}-1},\\
\alpha^u_k + v\alpha^u_{k-1}+v\gamma^u_{k,\tilde{k}} +v^2 \gamma^u_{k-1,\tilde{k}}, & j_{\tilde{k}-1} \pre i_{k-1}. 
\end{cases}\]

\item Suppose that $\beta^u_l$ is Kleshchev.  Let $\tilde{l}$ be maximal such that $i_{\tilde{l}}<j_l$.  Then
\[P(\beta^u_l) = \begin{cases}
\beta^u_l + v \alpha^u_{\tilde{l}} + v^2 \gamma^u_{\tilde{l},l}, & j_{l-1} \pre  i_{\tilde{l}}, \\
\beta^u_l+v\beta^u_{l-1} + v\gamma^u_{\tilde{l},l-1} + v^2\gamma^u_{\tilde{l},l}, & i_{\tilde{l}} \pre j_{l-1}.
\end{cases}\]
\item Suppose that $\gamma^u_{k,l}$ is Kleshchev.  Then
\[P(\gamma^u_{k,l}) = \begin{cases}
\gamma^u_{k,l} + v\gamma^u_{k-1, l} + v\gamma^u_{k, l+1} + v^2\gamma^u_{k-1, l+1}, & i_{k} \pre j_{l} \text{ or } j_{l+1} \pre  i_{k-1}, \\ 
\gamma^u_{k,l} + v\gamma_{k, l+1}^u + v\gamma^u_{k-1, l} + v\beta^u_{l+1} + v^2 \alpha^u_{k-1} , & j_l \pre  i_{k-1} \pre j_{l+1} \pre i_k, \\ 
\gamma^u_{k,l} + v\gamma^u_{k-1, l} +v \alpha^u_{k} + v^2\alpha^u_{k-1} , & j_l \pre i_{k-1} \pre i_{k} \pre j_{l+1}, \\
\gamma^u_{k,l} + v\gamma^u_{k, l+1}+v\beta^u_{l+1}+ v^2\beta^u_l , &  i_{k-1}\pre j_l \pre j_{l+1} \pre i_k, \\
\gamma^u_{k,l} + v\alpha^u_k + v^2\beta^u_l, & i_{k-1} \pre j_l \pre i_k \pre j_{l+1}.\\
\end{cases}\]
\end{itemize}
\item Suppose that $u=4$ or $u=5$.  
\begin{itemize}
\item Suppose that $\alpha^u_k$ is Kleshchev.   Let $\tilde{k}$ be maximal such that $j_{\tilde{k}} \pre i_k$.  Then
\[P(\alpha^u_k) = 
\begin{cases}
\alpha^u_k + v\beta^u_{\tilde{k}+1} + v\gamma^u_{k, \tilde{k}} + v\gamma^u_{k+1, \tilde{k}+1} + v^2 \gamma^u_{k+1, \tilde{k}}, & j_{\tilde{k}+1} \pre i_{k+1}, \\
\alpha^u_k + v\alpha^u_{k+1}+v\gamma^u_{k,\tilde{k}} + v^2 \gamma^u_{k+1,\tilde{k}}, &  i_{k+1} \pre j_{\tilde{k}+1}. 
\end{cases}\]
\item Suppose that $\beta^u_l$ is Kleshchev.  Let $\tilde{l}$ be minimal such that $j_l \pre i_{\tilde{l}}$.  Then
\[P(\beta^u_l) = \begin{cases}
\beta^u_l + v\alpha^u_{\tilde{l}} + v^2 \gamma^u_{\tilde{l},l}, &  i_{\tilde{l}} \pre j_{l+1}, \\
\beta^u_l+v\beta^u_{l+1} + v\gamma^u_{\tilde{l},l+1}+ v^2 \gamma^u_{\tilde{l},l} , & j_{l+1} \pre i_{\tilde{l}}.

\end{cases}\]
\item Suppose that $\gamma^u_{k,l}$ is Kleshchev.  

\[P(\gamma^u_{k,l}) = \begin{cases}
\gamma^u_{k,l} + v\gamma^u_{k+1, l} + v\gamma^u_{k, l-1} + v^2\gamma^u_{k+1, l-1}, & i_{k+1} \pre j_{l-1} \text{ or } j_l \pre  i_k, \\ 
\gamma^u_{k,l} + v\gamma_{k, l-1}^u + v\gamma^u_{k+1, l} + v\beta^u_{l-1} +v^2 \alpha^u_{k+1}, & i_k \pre j_{l-1} \pre i_{k+1} \pre j_l, \\ 
\gamma^u_{k,l} + v\gamma^u_{k, l-1} + v\beta^u_{l-1} + v^2 \beta^u_{l}, & i_k \pre j_{l-1} \pre j_l \pre i_{k+1}, \\
\gamma^u_{k,l} + v\gamma^u_{k+1, l}+ v\alpha^u_k+v^2 \alpha^u_{k+1}, &  j_{l-1} \pre i_k \pre i_{k+1} \pre j_l, \\
\gamma^u_{k,l} + v\alpha^u_k + v^2\beta^u_l, & j_{l-1} \pre i_k \pre j_l \pre i_{k+1}.\\
\end{cases}\]
\end{itemize}

\item
\begin{itemize}
\item Suppose that $\alpha^6_k$ is Kleshchev.  Let $\tilde{k}$ be maximal such that $j_{\tilde{k}}<i_k$.  Then
\[P(\alpha^1_k) = 
\begin{cases}
\alpha^6_k  + v \gamma^6_{k  \tilde{k}} + v^2 \beta^6_{\tilde{k}}, & i_{k-1} \pre j_{\tilde{k}}, \\
\alpha^6_k + v \alpha^6_{k-1}+ v \gamma^6_{k,\tilde{k}} + v^2 \gamma^6_{k-1,\tilde{k}}, & 
j_{\tilde{k}} \pre i_{k-1}.
\end{cases}\]
\item Suppose that $\beta^6_l$ is Kleshchev.  Let $\tilde{l}$ be maximal such that $i_{\tilde{l}}<j_l$.  Then
\[P(\beta^6_l) = \begin{cases}
\beta^6_l + v \gamma^6_{\tilde{l},l}  + v ^2\alpha^6_{\tilde{l}}, & 
j_{l-1} \pre i_{\tilde{l}}, \\
\beta^6_l+v\beta^6_{l-1} + v\gamma^6_{\tilde{l},l} + v^2\gamma^6_{\tilde{l},l-1}, &
i_{\tilde{l}} \pre j_{l-1}. 
\end{cases}\]
\item Suppose that $\gamma^6_{k,l}$ is Kleshchev.  Then
\[P(\gamma^6_{k,l}) = \begin{cases}
\gamma^6_{k,l} + v \gamma^6_{k,l-1} + v\gamma^6_{k-1,l} + v^2 \gamma^6_{k-1,l-1}, &  i_{k} \pre j_{l-1} \text{ or } j_{l} \pre i_{k-1}, \\
\gamma^6_{k,l}+ v \gamma^6_{k-1,l} + v\alpha_{k}^6 + v\beta^6_{l-1}  + v^2 \gamma^6_{k-1,l-1},  
& i_{k-1} \pre j_{l-1} \pre i_{k} \pre j_{l}, \\
\gamma^6_{k,l} +v\gamma^6_{k,l-1} + v\beta^6_l + v^2\beta^1_{l-1}, 
&i_{k-1} \pre j_{l-1} \pre  j_{l} \pre i_{k}, \\
\gamma^6_{k,l} +v\gamma^6_{k-1,l} + v\alpha^6_k + v^2\alpha^6_{k-1}, 
& j_{l-1} \pre i_{k-1} \pre i_{k} \pre j_{l}, \\
\gamma^6_{k,l}+ v\gamma^6_{k,l-1} +v \alpha^6_{k-1} +v \beta^6_{l } + v^2\gamma^6_{k-1,l-1} ,  
& j_{l-1} \pre i_{k-1} \pre j_{l} \pre i_{k}. \\
\end{cases}\]
\end{itemize}
\end{enumerate}

\end{theorem}

Before proving this theorem, we give an example.  Recall that when we change the values of $u$, we essentially permute the rows of the matrix and hence the order of the terms in the multicharge and the components in the multipartitions.  From the example below, we can see that this leads to a non-trivial rearrangement of the decomposition matrices.    

\begin{ex} \label{Ex4}
Suppose that $e =4$.  Take the base tuple to be $B={\bf 0}$.
Below, for $u=1,3,5$ we have taken $Y=\{1,3\}$ and $Z=\{0,2\}$ and for $u=2,4,6$ we have taken $Y=\{0,2\}$ and $Z=\{1,3\}$.  Note that the multicharge is determined by the partitions.  The block decomposition matrices are then as follows, where the dimensions of the simple modules are given in the top row and the dimensions of the Specht modules are given in the last column.   For reasons of space, we have removed the outer brackets on each partition and used the symbol $|$ to indicate a new component.   
\begin{align*}
&\begin{array}{|c|l|cccc|c|} \hline 
\mc=(0,1,2) &&40&20&60&40& \\ \hline
\beta^{(1)}_1 & (\emptyset|2,1|1^3) & 1&&&&40 \\
\gamma^{(1)}_{11} & (1|2|1^3) & v&1&&&60 \\
\alpha^{(1)}_{1} & (1|2,1|1^2) & v^2&v&1&&120 \\
\gamma^{(1)}_{12} & (1|2^2|1) &&&v&&60 \\
\beta^{(1)}_2 & (2|2,1|1) & &v&v^2&1&120 \\
\gamma^{(1)}_{21} & (3|\emptyset|1^3) &&v&&&20 \\
\gamma^{(1)}_{22} & (3|1^2|1) &&v^2&&v&60 \\
\alpha^{(1)}_{2} & (3|2,1|\emptyset) & &&&v^2&40 \\ \hline
\end{array} 
&&\begin{array}{|c|l|cccc|c|} \hline 
\mc=(2,1,0) &&40&20&60&40& \\ \hline
\beta^{(6)}_{2}&(\emptyset|2,1|3) &1&&&&40 \\ 
\gamma^{(6)}_{22}&(1|1^2|3) &v&1&&&60 \\ 
\alpha^{(6)}_{2} & (1|2,1|2) & v^2&v&1&&120 \\
\gamma^{(6)}_{21} & (1|2^2|1) &&&v&&60 \\
\beta^{(6)}_{1} & (1^2|2,1|1) &&v&v^2&1&120 \\
\gamma^{(6)}_{12} & (1^3|\emptyset|3) & &v&&&20 \\
\gamma^{(6)}_{11} & (1^3|2|1) &&v^2&&v&60 \\
\alpha^{(6)}_{1} & (1^3|2,1|\emptyset) & &&&v^2&40 \\ \hline
\end{array}\\
\\
&\begin{array}{|c|l|cccc|c|} \hline 
\mc=(2,0,1) &&40&60&20&40& \\ \hline
\beta^{(2)}_{2}&(\emptyset|3|2,1) &1&&&&40 \\ 
\gamma^{(2)}_{21}&(1|1|2^2) &&1&&&60 \\ 
\alpha^{(2)}_{2} & (1|2|2,1) & v&v&1&&120 \\
\gamma^{(2)}_{22} & (1|3|1^2) &v^2 &&v&&60 \\
\beta^{(2)}_{1} & (1^2|1|2,1) &&v^2&v&1&120 \\
\alpha^{(2)}_{1} & (1^3|\emptyset|2,1) & &&&v&40 \\
\gamma^{(2)}_{11} & (1^3|1|2) &&&v&v^2&60 \\
\gamma^{(2)}_{12} & (1^3|3|\emptyset) & &&v^2&&20 \\ \hline
\end{array}
&&\begin{array}{|c|l|cccc|c|} \hline 
\mc=(0,2,1) &&40&60&20&40& \\ \hline
\beta^{(5)}_{1}&(\emptyset|1^3|2,1) &1&&&&40 \\ 
\gamma^{(5)}_{12}&(1|1|2^2) &&1&&&60 \\ 
\alpha^{(5)}_{1} & (1|1^2|2,1) & v&v&1&&120 \\
\gamma^{(5)}_{11} & (1|1^3|2) &v^2 &&v^2&&60 \\
\beta^{(5)}_{2} & (2|1|2,1) &&v^2&v&1&120 \\
\alpha^{(5)}_{4} & (3|\emptyset|2,1) & &&&v&40 \\
\gamma^{(5)}_{22} & (3|1|1^2) &&&v&v^2&60 \\
\gamma^{(5)}_{21} & (3|1^3|\emptyset) & &&v^2&&20 \\ \hline
\end{array} \\ \\
& 
\begin{array}{|c|l|cccc|c|} \hline 
\mc=(1,2,0) &&20&40&60&40& \\ \hline
\gamma^{(3)}_{21}&(\emptyset|1^3|3) &1&&&&20 \\ 
\gamma^{(3)}_{22}&(1^2|1|3) &v&1&&&60 \\ 
\gamma^{(3)}_{11} & (2|1^3|1) &v&&1&&60 \\
\alpha^{(3)}_2 & (2,1|\emptyset|3) & &v&&&40 \\
\beta^{(3)}_{2} & (2,1|1|2) & v&v^2&&1&120 \\
\alpha^{(3)}_1 & (2,1|1^2|1) &v^2 &&v&v&120 \\
\beta^{(3)}_{1} & (2,1|1^3|\emptyset) & &&v^2&&40 \\ 
\gamma^{(3)}_{12} & (2^2|1|1) &&&&v^2&60 \\ \hline
\end{array}&
&\begin{array}{|c|l|cccc|c|} \hline 
\mc=(1,0,2) &&20&40&60&40& \\ \hline
\gamma^{(4)}_{12}&(\emptyset|3|1^3) &1&&&&20 \\ 
\gamma^{(4)}_{11}&(2|1|1^3) &v&1&&&60 \\ 
\gamma^{(4)}_{22} & (1^2|3|1) &v&&1&&60 \\
\alpha^{(4)}_1 & (2,1|\emptyset|1^3) & &v&&&40 \\
\beta^{(4)}_{1} & (2,1|1|1^2) & v&v^2&&1&120 \\
\alpha^{(4)}_2 & (2,1|2|1) &v^2 &&v&v&120 \\
\beta^{(4)}_{2} & (2,1|3|\emptyset) & &&v^2&&40 \\ 
\gamma^{(4)}_{21} & (2^2|1|1) &&&&v^2&60 \\ \hline
\end{array} 
\end{align*}
\end{ex}

\begin{proof}[Proof of Theorem~\ref{MainTheorem}]
The proof of the main theorem is given by a case-by-case analysis in which we apply Corollary~\ref{PutTogether}. Rather than go through each one, we look in detail at some cases when $u=6$; we believe that this will sufficiently illustrate the techniques we use so that the reader can verify the remainder for themselves.   By Lemma~\ref{Ignore01}, we may assume that if $e\geq 2$ then $X,X' = \emptyset$ and if $e=0$ then if $c \in X$ then $c+1 \in X$ and if $c \in X'$ then $c-1 \in X'$.  

Suppose that $\alpha^6_k$ is Kleshchev so that $j_1 \pre i_k$.  Choose $\tilde{k}$ maximal such that $j_{\tilde{k}}\pre i_k$.  Suppose first that $i_{k-1} \pre j_{\tilde{k}}$. Then $j_{\tilde{k}}+1 = i_k$ and part of the matrix for $\gamma^6_{kl}$ is given by
\[\begin{array}{|DCCD|} \hline
\elp&j_{\tilde{k}} & i_k&\elp \\ \hline
&1&0& \\ 
&1&0& \\ 
&0&1& \\ \hline
\end{array}\; .\]
Let 
\[L=\begin{array}{|DCCD|} \hline
\elp &j_{\tilde{k}}& i_k &\elp \\ \hline
&1&0& \\ 
&1&0& \\ 
&1&0& \\ \hline
\end{array}\]
with all other columns as in $\alpha^6_k$.  Then $\wt(L)=0$ and 
\begin{align*}
G^{(1)}_{i_k} L &
=\begin{array}{|DCCD|} \hline 
\elp &j_{\tilde{k}}& i_k &\elp \\ \hline
&1&0& \\ 
&1&0& \\ 
&0&1& \\ \hline
\end{array} 
+ v \;\begin{array}{|DCCD|} \hline
\elp &j_{\tilde{k}}& i_k &\elp \\ \hline
&1&0& \\ 
&0&1& \\ 
&1&0& \\ \hline
\end{array} +v^2 \;
\begin{array}{|DCCD|} \hline
\elp &j_{\tilde{k}}& i_k &\elp \\ \hline
&0&1& \\ 
&1&0& \\ 
&1&0& \\ \hline
\end{array} 
\\
& =\alpha^6_{k}+ v \gamma^6_{\tilde{k}k}+v^2 \beta^6_{\tilde{k}}.  
\end{align*}
Suppose now that $j_{\tilde{k}} \pre i_{k-1}$.   Then part of the matrix for $\gamma^6_{kl}$ is given by
\[\begin{array}{|DCCDCCD|} \hline
\elp & j_{\tilde{k}} &&&i_{k-1}& i_k&\elp \\ \hline
&1&1&\elp&1&0& \\ 
&1&0&\elp&0&0& \\ 
&0&0&\elp&0&1& \\ \hline
\end{array}\;.\]
Let 
\[L=\begin{array}{|DCCDCCD|} \hline
\elp &j_{\tilde{k}} &&\elp&i_{k-1}& i_k& \elp \\ \hline
&1&1&\elp&1&0& \\ 
&1&0&\elp&0&0& \\ 
&1&0&\elp&0&0& \\ \hline
\end{array}\]
with all other columns as in $\alpha^6_k$.  Then $\wt(L)=0$ and 
\begin{align*}
G^1_{i_k} \elp G^1_{j_{\tilde{k}+2}}  G^1_{j_{\tilde{ k}+1}} L &=
\begin{array}{|DCCDCCD|} \hline
\elp&j_{\tilde{k}} &&\elp&i_{k-1}& i_k&\elp \\ \hline
&1&1&\elp&1&0& \\ 
&1&0&\elp&0&0& \\ 
&0&0&\elp&0&1& \\ \hline
\end{array}
+ v \; \begin{array}{|DCCDCCD|} \hline
\elp&j_{\tilde{k}} &&\elp&i_{k-1}& i_k&\elp \\ \hline
&1&1&\elp&0&1 &\\ 
&1&0&\elp&0&0 &\\ 
&0&0&\elp&1&0 &\\ \hline
\end{array} \\
&\qquad + v \; \begin{array}{|DCCDCCD|} \hline
\elp&j_{\tilde{k}} &&\elp&i_{k-1}& i_k&\elp \\ \hline
&1&1&\elp&1&0& \\ 
&0&0&\elp&0&1& \\ 
&1&0&\elp&0&0& \\ \hline
\end{array}
+v^2 \; \begin{array}{|DCCDCCD|} \hline
\elp&j_{\tilde{k}} &&\elp&i_{k-1}& i_k&\elp \\ \hline
&1&1&\elp&0&1& \\ 
&0&0&\elp&1&0& \\ 
&1&0&\elp&0&0& \\ \hline
\end{array}\\
& =\alpha^6_{k}+ v \alpha^6_{k-1} + v \gamma^6_{k \tilde{k}} + v^2 \gamma^6_{k-1 \tilde{k}}.
\end{align*}
The case where $\beta^6_l$ is Kleshchev is very similar, so we omit it here and finish by looking in detail at the Kleshchev multipartitions $\gamma^6_{kl}$ where $i_k < j_l$ and $k \neq 1$.  First suppose that $j_{l-1}<i_{k-1}<i_k<j_l$.  Then part of the matrix for $\gamma^6_{kl}$ is given by
\[\begin{array}{|DCCCDCCD|} \hline
\elp& i_{k-1}&i_k& &\elp&&j_l &\elp \\ \hline
&1&0&1&\elp&1&1& \\ 
&0&1&0&\elp&0&0& \\ 
&0&0&0&\elp&0&1& \\ \hline
\end{array}\;.\]
Let 
\[L=\begin{array}{|DCCCDCCD|} \hline
\elp& i_{k-1}&i_k& &\elp&&j_l &\elp \\ \hline
&1&0&1&\elp&1&1& \\ 
&1&0&0&\elp &0&0& \\ 
&1&0&0&\elp & 0&0& \\ \hline
\end{array}\]
with all other columns as in $\gamma^6_{kl}$.  Then $\wt(L)=0$ and 
\begin{align*}
G^{(1)}_{j_l} \ldots G^{(1)}_{i_{k+1}}G^{(2)}_{i_k} L & = 
\begin{array}{|DCCCDCCD|} \hline
\elp& i_{k-1}&i_k& &\elp&&j_l &\elp \\ \hline
&1&0&1&\elp&1&1& \\ 
&0&1&0&\elp&0&0& \\ 
&0&0&0&\elp&0&1& \\ \hline
\end{array}
+v \; \begin{array}{|DCCCDCCD|} \hline
\elp& i_{k-1}&i_k&&\elp&& j_l &\elp \\ \hline
&1&0&1&\elp&1&1& \\ 
&0&0&0&\elp&0&1& \\ 
&0&1&0&\elp&0&0& \\ \hline 
\end{array} \\
& \qquad +v \; \begin{array}{|DCCCDCCD|} \hline
\elp& i_{k-1}&i_k& &\elp&&j_l &\elp \\ \hline
&0&1&1&\elp&1&1& \\ 
&1&0&0&\elp&0&0& \\ 
&0&0&0&\elp&0&1& \\ \hline
\end{array}
+v^2 \; \begin{array}{|DCCCDCCD|} \hline
\elp& i_{k-1}&i_k& &\elp&&j_l &\elp \\ \hline
&0&1&1&\elp&1&1& \\ 
&0&1&0&\elp&0&0& \\ 
&1&0&0&\elp&0&0& \\ \hline
\end{array}
\\
&=\gamma^6_{kl}+v\alpha^6_k + v \gamma^6_{k-1l}+v^2 \alpha^6_{k-1}.
\end{align*}
Now suppose that $i_{k-1}<j_{l-1}<i_k<j_l$.  Then part of the matrix for $\gamma^6_{kl}$ is given by
\[\begin{array}{|DCCDCCCDCCD|} \hline
\elp &i_{k-1} &&&j_{l-1}&i_k&&\elp&&j_l& \elp \\ \hline
&1&1&\elp&1&0&1&\elp&1&1& \\ 
&0&1&\elp&1&1&0&\elp&0&0& \\ 
&0&0&\elp&0&0&0&\elp&0&1& \\ \hline
\end{array}\; .\]
Let \[L= \begin{array}{|DCCDCCCDCCD|} \hline
\elp &i_{k-1} &&&j_{l-1}&i_k&&\elp&&j_l& \elp \\ \hline
&1&1&\elp&1&0&1&\elp&1&1& \\ 
&1&1&\elp&1&0&0&\elp&0&0& \\ 
&0&0&\elp&1&0&0&\elp&0&0& \\ \hline
\end{array}\]
with all other columns as in $\gamma^6_{kl}$.  Then $\wt(L)=0$ and 
\begin{align*}
G^{(1)}_{i_{k-1}+1}&\ldots G^{(1)}_{j_{l-1}-1} G^{(1)}_{j_{l-1}}G^{(1)}_{j_l} \ldots G^{(1)}_{i_{k+1}}G^{(2)}_{i_k}L
= \begin{array}{|DCCDCCCDCCD|} \hline
\elp &i_{k-1} &&&j_{l-1}&i_k&&\elp&&j_l& \elp \\ \hline
&1&1&\elp&1&0&1&\elp&1&1& \\ 
&0&1&\elp&1&1&0&\elp&0&0& \\ 
&0&0&\elp&0&0&0&\elp&0&1& \\ \hline
\end{array} \\ 
&+v\; \begin{array}{|DCCDCCCDCCD|} \hline
\elp &i_{k-1} &&&j_{l-1}&i_k&&\elp&&j_l& \elp \\ \hline
&1&1&\elp&1&0&1&\elp&1&1& \\ 
&0&1&\elp&1&0&0&\elp&0&1& \\ 
&0&0&\elp&0&1&0&\elp&0&0& \\ \hline
\end{array} 
+v\; \begin{array}{|DCCDCCCDCCD|} \hline
\elp &i_{k-1} &&&j_{l-1}&i_k&&\elp&&j_l& \elp \\ \hline
&0&1&\elp&1&1&1&\elp&1&1& \\ 
&1&1&\elp&1&0&0&\elp&0&0& \\ 
&0&0&\elp&0&0&0&\elp&0&1& \\ \hline
\end{array} \\
&+v\; \begin{array}{|DCCDCCCDCCD|} \hline
\elp &i_{k-1} &&&j_{l-1}&i_k&&\elp&&j_l& \elp \\ \hline
&1&1&\elp&0&1&1&\elp&1&1& \\ 
&0&1&\elp&1&0&0&\elp&0&1& \\ 
&0&0&\elp&1&0&0&\elp&0&0& \\ \hline
\end{array}  
+v^2\; \begin{array}{|DCCDCCCDCCD|} \hline
\elp &i_{k-1} &&&j_{l-1}&i_k&&\elp&&j_l& \elp \\ \hline
&0&1&\elp&1&1&1&\elp&1&1& \\ 
&1&1&\elp&0&0&0&\elp&0&1& \\ 
&0&0&\elp&1&0&0&\elp&0&0& \\ \hline
\end{array} \\
&= \gamma^6_{kl} + v \alpha_k^6 + v \gamma^6_{k-1l}+v\beta^6_{l-1}+v^2 \gamma_{k-1l-1}.
\end{align*}
Finally, suppose that $i_{k-1}<i_k<j_{l-1}<j_l$.  Then part of the matrix for $\gamma^6_{kl}$ is given by
\[\begin{array}{|DCCDCCDCCDCCD|} \hline
\elp &i_{k-1} &&\elp&&i_k&\elp&j_{l-1}&&\elp&&j_l& \elp \\ \hline
&1&1&\elp&1&0&&1&1&\elp &1& 1&\\ 
&0&1&\elp&1&1&&1&0& \elp& 0&0& \\ 
&0&0&\elp&0&0&&0&0& \elp& 0& 1& \\ \hline
\end{array}\;.\]
Let \[L= \begin{array}{|DCCDCCDCCDCCD|} \hline
\elp &i_{k-1} &&\elp&&i_k&\elp&j_{l-1}&&\elp&&j_l& \elp \\ \hline
&1&1&\elp&1&0&&1&1&\elp &1& 1&\\ 
&1&1&\elp&1&0&&1&0& \elp& 0&0& \\ 
&0&0&\elp&0&0&&1&0& \elp& 0& 0& \\ \hline
\end{array}\]
with all other columns as in $\gamma^6_{kl}$. Then 
\[G^{(1)}_{j_l} \ldots G^{(1)}_{j_{l-1}+2} G^{(1)}_{j_{l-1}+1}G^{(1)}_{i_{k-1}+1}\ldots G^{(1)}_{i_k-1}G^{(1)}_{i_k}L 
= \gamma^6_{kl}+v\gamma^6_{kl-1} + v \gamma^6_{k-1l} + v^2 \gamma^6_{k-1l-1}. \]
\end{proof}

We now use Theorem~\ref{MainTheorem} to prove Theorem~\ref{MainTheorem2}.  Recall that if $\bmu$ and $\blam$ belong to the same  block then $\bmu \leadsto \blam$ if $\blam$ is formed from $\bmu$ by removing a single rim hook from component $k$ and attaching it to component $k+1$, where $k \in \{1,2,\ldots,r-1\}$, and if the leg lengths of these two hooks are equal then we write $\bmu \rel \blam$. 

\begin{ex} 
Let $e=4$ and let $\mc=(0,1,2)$.  Let $\blam=((1),(2),(1^3))$ and let $\bmu=((3),\emptyset,(1^3))$.  Then $\blam$ is formed from $\bmu$ by moving a hook from component 1 to component 2.  The attached hook and the removed hook both have leg length 0, so $\bmu\rel \blam$.  As we can see from the first table in Example~\ref{Ex4}, we also have $\widetilde{\blam} \gdom \bmu \gdom \blam$ and $d_{\bmu\blam}(v)=v$.   
\end{ex}

\begin{lemma} \label{LegLength}
Suppose that $\blam = \Pt(B,M)$ and $\bmu = \Pt(B,L)$ are $3$-multipartitions in an indecomposable core block of weight 2. 
\begin{enumerate}
\item $\bmu \leadsto \blam$ if and only if there exist $i, j \in I$ with $i<j$ and $k \in \{1,2 \}$ such that 
\begin{align*}
L(k,i)=L(k+1,j)=M(k,j)=M(k+1,i)&=0, \\
L(k,j)=L(k+1,i)=M(k,i)=M(k+1,j)&=0, \\
\end{align*}
and $L(k',i')=M(k',i')$ for all possible $k' \neq k,k+1$ and $i' \neq i,j$.  
\item $\bmu \rel \blam$ if and only if $\bmu \leadsto \blam$ and, for the $i$ and $j$ specified in part~(i), we have 
$$ | \{ x \ | \ i < x < j, L(k,x) = 1 \} | = | \{ x \ | \ i<x<j, L(k+1,x) = 1 \} |. $$
\end{enumerate}

\end{lemma}

\begin{proof}
The condition in part (i) says that row $k+1$ of $M$ is obtained from row $k+1$ of $L$ by moving a $1$ to the right from column $i$ to column $j$, into a position occupied by a $0$. In terms of abaci, this means removing a bead from the bottom of one runner of $\bmu^{(k+1)}$ and adding a bead to the bottom of another runner. Although the ordering $\prec$ will determine precisely which runners of the abacus display are affected, the condition $i<j$ does require that the bead that is added will have a higher $\beta$-number. Consequently, one can regard $\blam^{(k+1)}$ as being obtained from $\bmu^{(k+1)}$ by moving a bead a number of spaces to the right (with the usual convention that moving a bead to the right from runner $e-1$ takes it to runner $0$ of the next row down). This corresponds precisely to adding a rim hook to $\bmu^{(k+1)}$: the number of boxes added is equal to the number of spaces the bead is moved (which will depend on $B$). Similarly, the condition in part (i) says that $\blam^{(k)}$ is obtained from $\bmu^{(k)}$ by removing a rim hook, so if the condition is satisfied then $\bmu \leadsto \blam$. 

In general, if $\blam$ is obtained from $\bmu$ by removing a rim hook from $\bmu^{(k)}$ and adding a rim hook of the same size to $\bmu^{(k+1)}$, then the corresponding abaci could be related by moving beads between different runners in $\bmu^{(k)}$ than $\bmu^{(k+1)}$ (that is, by moving $1$s between different columns in row $k$ of $L$ than in row $k+1$) or by moving beads that are not the lowest beads on their runners (in which case $\blam$ would no longer be a multicore). However, since we are assuming that $\blam$ and $\bmu$ lie in the same block, the bead shift in component $k+1$ must involve the same columns as the bead shift in component $k$ in order to satisfy Proposition~\ref{Block}. Moreover, since we are assuming that the block is a core block, the only beads we can move in $\bmu$ are those that are the lowest on their respective runners. So, if $\bmu \leadsto \blam$ then the condition in part (i) must hold. 

The condition in part (ii) involves counting the number of $1$s between column $i$ and column $j$ of the matrix $L$, in both rows $k$ and $k+1$, which is a number that is connected to the leg length of the rim hooks that we are adding and removing. When we move a bead to the right in the abacus display for $\bmu^{(k+1)}$, thereby adding a rim hook, the leg length of that rim hook is equal to the number of beads past which we move. (A similar statement holds for the leg length of the rim hook removed from $\bmu^{(k)}$.) While this number will be partly determined by $B$, it is at least the case that the way $B$ is used to determine abaci is the same in $\bmu^{(k)}$ as it is in $\bmu^{(k+1)}$, and so the base tuple's contribution to the leg length of the hooks that are added/removed will be the same in components $k$ and $k+1$. Additional beads encountered as we move to the right in $\bmu^{(k+1)}$ will comprise some of those extra beads added to the base abacus to construct $\bmu^{(k+1)}$ (which is to say those beads encoded by $1$s in row $k+1$ of the matrix $L$). Specifically, because of the way that the ordering $\prec$ relates the columns of $L$ to the runners on the abacus, the beads that we encounter will be precisely those that correspond to $1$s between column $i$ and column $j$ of row $k+1$ of $L$. Since a similar statement holds for row $k$, it follows that $\bmu \rel \blam$ if and only if rows $k$ and $k+1$ of $L$ contain an equal number of $1$s between column $i$ and column $j$, which is precisely the condition in part (ii) of the lemma.   \end{proof}

 If $\blam$ is a Kleshchev multipartition, recall that $\widetilde{\blam}=(\blam^{\diamond})'$ is the conjugate of the image of $\blam$ under the generalised Mullineux involution. 
Theorem~\ref{MainTheorem2} was as follows.  

\begin{theorem} \label{MT2}
Suppose that $\blam$ and $\bmu$ are $3$-multipartitions in an indecomposable core block of weight 2 and that $\blam$ is a Kleshchev multipartition.  Then
\[d_{\bmu\blam}(v)=\begin{cases}
1, & \bmu = \blam, \\
v, & \widetilde{\blam} \gdom \bmu \gdom \blam \text{ and } (\bmu \rel \blam \text{ or } \widetilde{\blam} \rel \bmu), \\
v^2, & \bmu = \widetilde{\blam}, \\
0, & \text{otherwise}.  
\end{cases}\]
\end{theorem}

The result will follow from a case-by-case analysis of Theorem~\ref{MainTheorem}.  First recall the following result.

\begin{theorem} [{\cite[Corollary 2.4]{Fayers:WeightII},\cite[Remark~3.19]{BK:GradedDecomp}}] \label{MullConj}
Suppose that $\blam$ and $\bmu$ are multipartitions in the same block of weight $w$ and that $\blam$ is a Kleshchev multipartition. 
\begin{enumerate}
\item We have $d_{\blam\blam}(v)=1$ and $d_{\widetilde{\blam}\blam}(v)=v^{w}$.
\item We have $d_{\bmu \blam}(v)=0$ unless $\widetilde{\blam} \gedom \bmu \gedom \blam$.
\end{enumerate}
\end{theorem}

Suppose that $\blam$ is a Kleshchev 3-multipartition in an indecomposable core block of weight 2.  From Theorem~\ref{MullConj} and Theorem~\ref{MainTheorem} we can read off $\widetilde{\blam}$, and we know that $d_{\bmu\blam}(v)=0$ unless $\widetilde{\blam} \gedom \bmu \gedom \blam$, where $d_{\blam\blam}(v)=1$ and $d_{\widetilde{\blam}\blam}(v)=v^{2}$.  So to prove Theorem~\ref{MT2}, we take a set containing all partitions $\btau$ such that $\widetilde{\blam} \gdom \btau \gdom \blam$ and use Lemma~\ref{LegLength} to show that the partitions $\bmu$ in that set satisfy either 
$\bmu \rel \blam$ or $\widetilde{\blam} \rel \bmu$ if and only if the coefficient of $\bmu$ in $P(\blam)$ is $v$.  We remark that it turns out to be the case that if $\bmu \leadsto \blam$ and $\widetilde{\blam} \leadsto \bmu$ then $\bmu \rel \blam$ if and only if $\widetilde{\blam} \rel \bmu$.

Below, we give the details for the cases $u=2$ and $u=3$.  The notation is that used in Theorem~\ref{MainTheorem}.  We leave verification of the other four cases as an exercise for the reader.    

\begin{proof}[Proof of Theorem~\ref{MT2} when $u=2$ and $u=3$]
Suppose that $u=2$ or $u=3$.  Recall that the partitions are given by:  
\[\begin{array}{c|ccccc|c|cccc|c|cccc|} \cline{2-6} \cline{8-11} \cline{13-16}
&x&y&z&1_y&1_z&& x&y+1&z&1&& x&y&z+1&1 \\ \cline{2-6} \cline{8-11} \cline{13-16}
\gamma^2=&0&1&1&0&1 & \beta^2=& 0&1&1&0 &\alpha^2=& 0&1&1&0 \\
&0&0&0&0&1 & & 0&0&0&1 && 0&0&0&1 \\
&0&0&1&1&0 && 0&0&1&1 && 0&0&1&0\\
\cline{2-6} \cline{8-11} \cline{13-16}
\end{array}\]
\[\begin{array}{c|ccccc|c|cccc|c|cccc|} \cline{2-6} \cline{8-11} \cline{13-16}
&x&z&y&1_z&1_y&& x&z+1&y&1&& x&z&y+1&1 \\ \cline{2-6} \cline{8-11} \cline{13-16}
\gamma^3=&0&0&1&1&0 & \alpha^3=& 0&0&1&1 &\beta^3=& 0&0&1&0\\
&0&1&1&0&1 && 0&1&1&0 && 0&1&1&0 \\ 
&0&0&0&0&1 & & 0&0&0&1 && 0&0&0&1 \\\cline{2-6} \cline{8-11} \cline{13-16}
\end{array}\]

The relations in the following two lemmas are easy to check.  

\begin{lemma} \label{URelns1}
We have the following relations:
\begin{itemize}
\item $\alpha^u_k \gdom \alpha^u_{k+1}$ and $\beta^u_l \gdom \beta^u_{l+1}$.  
\item If $i_k < j_l$ then $\alpha^u_k \gdom \beta^u_l$ and if $j_l < i_k$ then $\beta^u_l \gdom \alpha^u_k$.
\item $\gamma^u_{k,l+1} \gdom \gamma^u_{k,l}\gdom \gamma^u_{k+1,l}$ and if $\gamma^u_{k,l}\gedom \gamma^u_{k',l'}$ then $k \leq k'$.
\item If $\gamma^u_{k,l}\gdom \alpha^u_v$ then $k \leq v$ and if $\alpha^u_v \gdom \gamma^u_{k,l}$ then $v \leq k$.
\item If $\gamma^u_{k,l} \gdom \alpha^u_k$ then $i_k<j_l$ and if $\alpha^u_k \gdom \gamma^u_{k,l}$ then $j_l < i_k$.    
\item If $\gamma^u_{k,l}\gdom \beta^u_w$ then $i_k < j_w$ and if $\beta^u_w \gdom \gamma^u_{k,l}$ then $j_w < i_k$.  
\end{itemize}
\end{lemma}

\begin{lemma} \label{URelns2}
We have the following relations:
\begin{itemize}
\item If $k<k'$ then $\alpha^u_k \leadsto \alpha^u_{k'}$ and if 
$l<l'$ then $\beta^u_l \leadsto \beta^u_{l'}$.
\item If $i_k < j_l$ then $\alpha^u_k \leadsto \beta^u_l$ and if $j_l < i_k$ then $\beta^u_l \leadsto \alpha^u_k$.
\item We have $\gamma^u_{k,l} \leadsto \alpha^u_v$ if and only if $v=k$ and $i_k < j_l$ and we have $\alpha^u_v \leadsto \gamma^u_{k,l}$ if and only if $v=k$ and $j_l < i_k$.
\item If $l'<l$ then $\gamma^u_{k,l} \leadsto \gamma^u_{k,l'}$ and if $k<k'$ then $\gamma^u_{k,l} \leadsto \gamma^u_{k',l}$.  
\item We have $\gamma^u_{k,l}\not\leadsto \beta^u_w$ and $\beta^u_w \not\leadsto \gamma^u_{k,l}$ in all cases.  
\end{itemize}
\end{lemma}

We now look at all the Kleshchev multipartitions in the block.  

\begin{enumerate}
\item Suppose that $\blam = \alpha^u_k$ is Kleshchev. Let $\tilde{k}$ be minimal such that $i_k \pre j_{\tilde{k}}$.  
\begin{itemize}
\item Suppose $i_{k-1} \pre j_{\tilde{k}-1}$.  Then $\widetilde{\blam}=\gamma^u_{k-1,\tilde{k}}$.  From Lemma~\ref{URelns1} the only partitions that can have the property that $\widetilde{\blam} \gdom \bmu \gdom \blam$ belong to the set
\[\{\alpha_{k-1}\} \cup \{\beta_l \mid i_{k-1}<j_l<i_k\} \cup \{\gamma_{k-1,l}\mid l<\tilde{k}\} \cup \{\gamma_{k,l} \mid i_k<j_l\}.\]
\[\begin{array}{|ll|cc|cc|c|} \hline 
\bmu& & \bmu \leadsto \alpha^u_k & \bmu \rel\alpha^u_k& \gamma^u_{k-1,\tilde{k}} \leadsto \bmu &\gamma^u_{k-1,\tilde{k}}\rel \bmu & 
\raisebox{0pt}[13pt][0pt]{$\bmu \rel \blam / \widetilde{\blam} \rel \bmu$} \\ \hline 
\renewcommand{\arraystretch}{2} 
\alpha_{k-1}& & \checkmark & \times & \checkmark & \times &\times \\ 
\beta_l &:i_{k-1}<j_l<\tilde{k}-1 & \checkmark & \times &\times && \times \\
\beta_{\tilde{k}-1} && \checkmark & \checkmark & \times & &\checkmark \\
\gamma_{k-1,l} &:l<\tilde{k}-1 &\times &&\checkmark &\times & \times \\
\gamma_{k-1,\tilde{k}-1} && \times && \checkmark &\checkmark& \checkmark \\
\gamma_{k,l} &:i_k<j_{\tilde{k}} & \checkmark & \times &\times&& \times \\ 
\gamma_{k,\tilde{k}} && \checkmark & \checkmark & \checkmark & \checkmark & \checkmark\\ \hline 
\end{array}\]

\item Suppose $j_{\tilde{k}-1} \pre i_{k-1}$.  Then $\widetilde{\blam}=\gamma^u_{k-1,\tilde{k}}$.  From Lemma~\ref{URelns1} the only partitions that can have the property that $\widetilde{\blam} \gdom \bmu \gdom \blam$ belong to the set
\[\{\alpha^u_{k-1}\} \cup \{\gamma^u_{k-1,l}\mid l< \tilde{k}\} \cup \{\gamma^u_{k,l} \mid i_k<j_l\}.\]
\[\begin{array}{|ll|cc|cc|c|} \hline 
\bmu& & \bmu \leadsto \alpha^u_k & \bmu \rel \alpha^u_k& \gamma^u_{k-1,\tilde{k}} \leadsto \bmu & \gamma^u_{k-1,\tilde{k}}\rel \bmu & 
\raisebox{0pt}[13pt][0pt]{$\bmu \rel \blam / \widetilde{\blam} \rel \bmu$} \\ \hline
\alpha_{k-1} && \checkmark & \checkmark & \checkmark&\checkmark &\checkmark \\
\gamma_{k-1,l}&: l<\tilde{k} &\times &&\checkmark & \times & \times \\
\gamma_{k,l}&: i_k<j_{\tilde{k}} & \checkmark & \times &\times&& \times \\ 
\gamma_{k,\tilde{k}} && \checkmark & \checkmark & \checkmark & \checkmark & \checkmark\\ \hline 
\end{array}\]
\end{itemize}
\item Suppose that $\blam = \beta^u_l$ is Kleshchev. Let $\tilde{l}$ be maximal such that $i_{\tilde{l}} \pre j_{l}$.  
\begin{itemize}
\item Suppose $j_{l-1} \pre i_{\tilde{l}}$.  Then $\widetilde{\blam}=\gamma^u_{\tilde{l},l}$.  From Lemma~\ref{URelns1} the only partition that can have the property that $\widetilde{\blam} \gdom \bmu \gdom \blam$ is $\alpha^u_{\tilde{l}}$.  
\[\begin{array}{|ll|cc|cc|c|} \hline 
\bmu& & \bmu \leadsto \beta^u_l & \bmu\rel\beta^u_l& \gamma^u_{\tilde{l},l} \leadsto \bmu & \gamma^u_{\tilde{l},l}\rel\bmu &
\raisebox{0pt}[13pt][0pt]{$\bmu \rel \blam / \widetilde{\blam} \rel \bmu$}
\\ \hline
\alpha^u_{\tilde{l}} & &\checkmark & \checkmark &\checkmark & \checkmark & \checkmark \\ \hline
\end{array}\]
\item Suppose that $i_{\tilde{l}}\pre j_{l-1}$.  Then $\widetilde{\blam}=\gamma^u_{\tilde{l},l}$.  From Lemma~\ref{URelns1} the only partitions that can have the property that $\widetilde{\blam} \gdom \bmu \gdom \blam$ 
belong to the set
\[\{\alpha^u_{\tilde{l}}\} \cup \{\beta^u_{w}\mid i_{\tilde{l}}<j_w<j_l\} \cup \{\gamma^u_{\tilde{l},w} \mid w<l\}.
\]
\[\begin{array}{|ll|cc|cc|c|} \hline 
\bmu& & \bmu \leadsto \beta^u_l & \bmu\rel\beta^u_l& \gamma^u_{\tilde{l},l} \leadsto \bmu & \gamma^u_{\tilde{l},l}\rel \bmu & 
\raisebox{0pt}[13pt][0pt]{$\bmu \rel \blam / \widetilde{\blam} \rel \bmu$} \\ \hline
\alpha^u_{\tilde{l}}& & \checkmark & \times  &\checkmark & \times & \times \\ 
\beta^u_w & : i_{\tilde{l}}<j_w <j_{l-1} & \checkmark & \times &\times& & \times \\
\beta^u_{l-1} &&\checkmark & \checkmark & \times &&\checkmark \\
\gamma^u_{\tilde{l},w} & : w<l-1 & \times && \checkmark & \times & \times \\
\gamma^u_{\tilde{l},l-1} &&\times &&\checkmark & \checkmark & \checkmark \\ \hline 
\end{array}\]
\end{itemize}
\item Suppose that $\blam=\gamma^u_{k,l}$ is Kleshchev.
\begin{itemize}
\item Suppose that $i_{k-1}\pre i_k\pre j_l \pre j_{l+1}$.  
Then $\widetilde{\blam}=\gamma^u_{k-1,l+1}$.  From Lemma~\ref{URelns1} the only partitions that can have the property that $\widetilde{\blam} \gdom \bmu \gdom \blam$ 
belong to the set
\[\{\alpha^u_{k-1}\} \cup \{\beta^u_{w}\mid i_{k-1}<j_w<i_k\} \cup \{\gamma^u_{k-1},w\} \mid w<l+1\} \cup \{\gamma_{k,w}\mid l<w\}.
\]
However, further investigation shows that if $i_{k-1}<j_w<i_k$ then $\beta_w \not \gdom \gamma_{k,l}$. 

\[\begin{array}{|ll|cc|cc|c|} \hline 
\bmu& & \bmu \leadsto \gamma^u_{k,l} & \bmu\rel \gamma^u_{k,l}& \gamma^u_{k-1,l+1} \leadsto \bmu & \gamma^u_{k-1,l+1}\rel \bmu &
\raisebox{0pt}[13pt][0pt]{$\bmu \rel \blam / \widetilde{\blam} \rel \bmu$}\\ \hline
\alpha^u_{k-1} &&\times && \checkmark & \times & \times \\
\gamma^u_{k-1,w} & :w<l &  \times&& \checkmark & \times & \times \\
\gamma^u_{k-1,l} & & \checkmark&\checkmark&\checkmark&\checkmark & \checkmark \\
\gamma^u_{k,l+1} & & \checkmark&\checkmark&\checkmark&\checkmark&\checkmark \\
\gamma^u_{k,w} & :w>l+1 & \checkmark & \times & \times& &\times \\ \hline   
\end{array}\] 
\item Suppose that $j_l \pre j_{l+1}\pre i_{k-1}\pre i_k$.  
Then $\widetilde{\blam}=\gamma^u_{k-1,l+1}$.  From Lemma~\ref{URelns1} the only partitions that can have the property that $\widetilde{\blam} \gdom \bmu \gdom \blam$ 
belong to the set
\[\{\alpha^u_{k}\} \cup \{\beta^u_{w}\mid i_{k-1}<j_w<i_k\} \cup \{\gamma^u_{k-1},w\} \mid w<l+1\} \cup \{\gamma_{k,w}\mid l<w\}.
\]
However, further investigation shows that if $i_{k-1}<j_w<i_k$ then $\gamma_{k-1,l+1} \not \gdom \beta_w$. 

\[\begin{array}{|ll|cc|cc|c|} \hline 
\bmu& & \bmu \leadsto \gamma^u_{k,l} & \bmu\rel \gamma^u_{k,l}& \gamma^u_{k-1,l+1} \leadsto \bmu & \gamma^u_{k-1,l+1}\rel \bmu & 
\raisebox{0pt}[13pt][0pt]{$\bmu \rel \blam / \widetilde{\blam} \rel \bmu$}
\\ \hline
\alpha^u_{k} && \checkmark & \times & \times &&\times \\
\gamma^u_{k-1,w} & :w<l & \times&& \checkmark & \times & \times \\
\gamma^u_{k-1,l} & & \checkmark&\checkmark&\checkmark&\checkmark & \checkmark \\
\gamma^u_{k,l+1} & & \checkmark&\checkmark&\checkmark&\checkmark&\checkmark \\
\gamma^u_{k,w} & :w>l+1 & \checkmark & \times &\times & &\times \\ \hline   
\end{array}\] 
\item Suppose that $j_l\pre i_{k-1}\pre j_{l+1} \pre i_k$.
Then $\widetilde{\blam}=\alpha_{k-1}$.  From Lemma~\ref{URelns1} the only partitions that can have the property that $\widetilde{\blam} \gdom \bmu \gdom \blam$ 
belong to the set
\[\{\alpha^u_{k}\} \cup \{\beta^u_{w}\mid i_{k-1}<j_w<i_k\} \cup \{\gamma^u_{k-1},w\} \mid j_w<i_{k-1}\} \cup \{\gamma_{k,w}\mid l<w\}.
\]
\[\begin{array}{|ll|cc|cc|c|} \hline 
\bmu& & \bmu \leadsto \gamma^u_{k,l} & \bmu\rel \gamma^u_{k,l}& \alpha_{k-1} \leadsto \bmu & \alpha_{k-1}\rel \bmu & 
\raisebox{0pt}[13pt][0pt]{$\bmu \rel \blam / \widetilde{\blam} \rel \bmu$}
\\ \hline
\alpha^u_{k} &&\checkmark&\times& \checkmark & \times & \times \\
\beta^u_{l+1} && \times && \checkmark & \checkmark & \checkmark \\ 
\beta^u_{w} & :j_{l+1}<j_w<i_k & \times&& \checkmark & \times & \times \\
\gamma^u_{k-1,w} & : j_w<i_{k-1} & \times&& \checkmark & \times & \times \\
\gamma^u_{k-1,l} & &  \checkmark&\checkmark&\checkmark&\checkmark & \checkmark \\ 
\gamma^u_{k,l+1} &  &  \checkmark&\checkmark&\times& & \checkmark \\ 
\gamma^u_{k,w} & :l+1<w & \checkmark & \times &\times& &\times \\ \hline
\end{array}\]
\item Suppose that $j_l\pre i_{k-1}\pre i_k \pre j_{l+1}$.
Then $\widetilde{\blam}=\alpha_{k-1}$.  From Lemma~\ref{URelns1} the only partitions that can have the property that $\widetilde{\blam} \gdom \bmu \gdom \blam$ 
belong to the set
\[\{\alpha^u_{k}\} \cup \{\gamma^u_{k-1},w\} \mid j_w<i_{k-1}\} \cup \{\gamma_{k,w}\mid l<w\}.
\]
\[\begin{array}{|ll|cc|cc|c|} \hline 
\bmu& & \bmu \leadsto \gamma^u_{k,l} & \bmu\rel \gamma^u_{k,l}& \alpha_{k-1} \leadsto \bmu & \alpha_{k-1}\rel \bmu & 
\raisebox{0pt}[13pt][0pt]{$\bmu \rel \blam / \widetilde{\blam} \rel \bmu$}
\\ \hline
\alpha^u_{k} &&\checkmark&\checkmark& \checkmark & \checkmark & \checkmark \\
\gamma^u_{k-1,w} & : j_w<i_{k-1} &\times && \checkmark & \times & \times \\
\gamma^u_{k-1,l} & &  \checkmark&\checkmark&\checkmark&\checkmark & \checkmark \\ 
\gamma^u_{k,w} & : l<w &\checkmark&\times &\times&& \times \\ \hline
\end{array}\]
\item Suppose that $i_{k-1}\pre j_l\pre j_{l+1} \pre i_k$.
Then $\widetilde{\blam}=\beta_{l}$.  From Lemma~\ref{URelns1} the only partitions that can have the property that $\widetilde{\blam} \gdom \bmu \gdom \blam$ 
belong to the set
\[\{\alpha^u_{k}\} \cup \{\beta^u_{w}\mid j_l < j_w <i_k\} \cup \{\gamma_{k,w}\mid l<w\}.
\]
\[\begin{array}{|ll|cc|cc|c|} \hline 
\bmu& & \bmu \leadsto \gamma^u_{k,l} & \bmu\rel \gamma^u_{k,l}& \beta_l \leadsto \bmu & \beta_l\rel \bmu &
\raisebox{0pt}[13pt][0pt]{$\bmu \rel \blam / \widetilde{\blam} \rel \bmu$}
\\ \hline
\alpha^u_{k} &&\checkmark&\times& \checkmark & \times & \times \\
\beta_{l+1} &&\times && \checkmark & \checkmark & \checkmark \\
\beta_{w} & :l+1<w & \times & &\checkmark &\times & \times \\ 
\gamma_{k,l+1} && \checkmark & \checkmark & \times && \checkmark \\
\gamma_{k,w} &: l+1<w &\checkmark &\times&\times & & \times \\ \hline   
\end{array}\]

\item Suppose that $i_{k-1}\pre j_l\pre i_k \pre j_{l+1}$.
Then $\widetilde{\blam}=\beta_{l}$.  From Lemma~\ref{URelns1} the only partitions that can have the property that $\widetilde{\blam} \gdom \bmu \gdom \blam$ 
belong to the set
\[\{\alpha^u_{k}\}\cup \{\gamma_{k,w}\mid l<w\}.
\]
\[\begin{array}{|ll|cc|cc|c|} \hline 
\bmu& & \bmu \leadsto \gamma^u_{k,l} & \bmu\rel \gamma^u_{k,l}& \beta_l \leadsto \bmu & \beta_l \rel \bmu & 
\raisebox{0pt}[13pt][0pt]{$\bmu \rel \blam / \widetilde{\blam} \rel \bmu$}
\\ \hline
\alpha^u_{k} &&\checkmark&\checkmark& \checkmark & \checkmark & \checkmark \\
\gamma_{k,w} &: l<w &\checkmark &\times &\times & & \times \\ \hline   
\end{array}\]
\end{itemize}
\end{enumerate}
This concludes the proof of Theorem~\ref{MT2} when $u=2$ and $u=3$.
\end{proof}

\end{document}